\documentclass[final,hidelinks,onefignum,onetabnum]{siamart220329}
\usepackage{amsfonts}
\usepackage{graphicx}
\usepackage{epstopdf}
\usepackage{algorithmic}
\usepackage{amssymb}
\usepackage{mathtools}
\usepackage{multirow}
\usepackage{array}
\usepackage[caption=false]{subfig}

\usepackage{enumerate}

\usepackage{pgfplots}
\pgfplotsset{width=10cm,compat=1.9}

\usepackage{tikz}

\ifpdf
  \DeclareGraphicsExtensions{.eps,.pdf,.png,.jpg}
\else
  \DeclareGraphicsExtensions{.eps}
\fi

\newsiamremark{remark}{Remark}
\newsiamremark{hypothesis}{Hypothesis}
\crefname{hypothesis}{Hypothesis}{Hypotheses}
\newsiamthm{claim}{Claim}

\headers{Spectral fractional Laplacian}{A. J. Salgado and S. E. Sawyer}

\title{A Semi-Analytic Diagonalization FEM for the Spectral Fractional Laplacian}

\author{Abner J. Salgado\thanks{Department of Mathematics, University of Tennessee, Knoxville
(\email{asalgad1@utk.edu}, \url{https://math.utk.edu/people/abner-j-salgado/}).}
\and Shane E. Sawyer\thanks{Department of Mathematics, University of Tennessee, Knoxville
(\email{sjw355@vols.utk.edu} ).}
}
\usepackage{amsopn}

\renewcommand{\imath}{{\mathfrak i}}

\ifpdf
\hypersetup{
  pdftitle={A Semi-Analytic Diagonalization FEM for the Spectral Fractional Laplacian},
  pdfauthor={A. J. Salgado and S. E. Sawyer}
}
\fi

\begin{document}

\maketitle

\begin{abstract}
We present a technique for approximating solutions to the spectral fractional Laplacian, which is based on the Caffarelli-Silvestre extension and diagonalization.
Our scheme uses the analytic solution to the associated
eigenvalue problem in the extended dimension.
We show its relation to a quadrature scheme.
Numerical examples demonstrate the performance of the method.
\end{abstract}



\section{Introduction}
\label{sec:intro}

Nonlocal and fractional order operators are of increasing interest in a variety of fields of study.
Application areas include image processing and denoising \cite{gatto2015numerical},
electroconvection \cite{bonito2020electroconvection},
quasi-geostrophic flow \cite{bonito2021numerical},
cardiac electrophysiology \cite{cusimano2018space, cusimano2021space},
and anomalous diffusion \cite{metzler2000}, to name but a very few of them.

In this work, we are interested in problems involving fractional powers of the Dirichlet Laplacian $(-\Delta)^s$,
with $s\in (0,1)$. Let the domain $\Omega$ be a convex, open, and bounded subset of $\mathbb{R}^d$, $d\geq 1$,
with Lipschitz boundary $\partial \Omega$, and our problem of interest is stated as: given $s\in (0,1)$ and a sufficiently smooth
$f$, find $u$ such that
\begin{equation} \label{eq:pde-problem}
    (-\Delta)^s u = f \qquad   \text{in } \Omega \, . \\
\end{equation}
The meaning of this operator will be specified in \Cref{sec:prelim}. At this stage it suffices to say that
the fractional Laplacian is a nonlocal operator, making it difficult to solve numerically. It has been shown by Caffarelli and
Silvestre \cite{caffarelli2007extension} that the solution of \eqref{eq:pde-problem} in $\mathbb{R}^d$ can be found by extending
the problem to the upper half space $\mathbb{R}_+^{d+1}$ via a Dirichlet-to-Neumann map. This extension method 
has been adapted to bounded domains, $\Omega \subset \mathbb{R}^d$, in \cite{cabre2010positive, stinga2010extension}.
This poses the extension problem on the semi-infinite cylinder $\mathcal{C} = \Omega \times (0,\infty)$.
The problem now reads as follows: find $\mathcal{U}$ such that
\begin{equation} \label{eq:extended-problem}
    \text{div } (y^\alpha \nabla \mathcal{U} ) = 0\,   \text{in } \mathcal{C} \, , \qquad
    \mathcal{U} = 0 \,   \text{on } \partial_L \mathcal{C} \, , \qquad
    \frac{\partial \mathcal{U}}{\partial \nu^\alpha} = d_s f \,   \text{on } \Omega \times \{ 0 \} \, ,
\end{equation}
where $\partial_L \mathcal{C} =\partial \Omega \times [0,\infty)$ is the lateral boundary of the cylinder,
and
\begin{equation*} \label{eq:conormal-derivative}
  \frac{\partial \mathcal{U}}{\partial \nu^\alpha} = - \lim_{y \downarrow 0} y^\alpha \mathcal{U}_y
\end{equation*}
denotes the conormal derivative of $\mathcal{U}$. Here, $\nu$ is the unit outer normal vector
to $\partial\mathcal{C}$ along $\Omega \times \{0\}$. The parameter $\alpha$ is related to $s$ by
\begin{equation} \label{eq:define-alpha}
  \alpha = 1 -2s \in (-1,1) \, .
\end{equation}
Lastly, $d_s$ is a positive normalization constant that depends only on the power $s$ via
\begin{equation} \label{eq:define-ds}
  d_s = \frac{2^{1-2s} \Gamma (1-s)}{\Gamma (s)} \, .
\end{equation}
As seen in \cite{caffarelli2007extension,cabre2010positive,stinga2010extension}, the solution to 
\eqref{eq:pde-problem} is related to 
\eqref{eq:extended-problem} by
\begin{equation*} \label{eq:frac-lap-to-dtn}
  d_s (-\Delta)^s u = \frac{ \partial \mathcal{U} }{\partial \nu^\alpha} \qquad \text{in } \Omega \, .
\end{equation*}
Numerical schemes have been introduced to approximate solutions of \eqref{eq:pde-problem} by computing solutions of \eqref{eq:extended-problem}. A piecewise linear FEM was introduced and analyzed in
\cite{nochetto2015pde}. This was improved in many ways in \cite{banjai2019tensor}.
The improvements include a tensor product formulation with \textit{hp}-FEM in the extended variable $y$ coupled
to $P_1$-FEM in $\Omega$. A diagonalization technique was also introduced that allows the Caffarelli-Silvestre
extension to decouple into independent, singularly perturbed second-order reaction-diffusion problems in $\Omega$.
Lastly, the use of \textit{hp}-FEM in $\Omega$, along with certain regularity assumptions, provided for exponential convergence
to $u$ in $\Omega$.

In this work we build on \cite{banjai2019tensor} and replace the \textit{hp}-FEM
scheme in the extended variable with the analytical solution to the one dimensional second-order eigenvalue problem induced by the diagonalization
technique. We call the resulting scheme an \textit{exact diagonalization technique} to distinguish it from earlier approaches.
By using the analytic solution, the instability inherent in the numerical approximation to a large eigenvalue problem
\cite{zhang2015many} can be mitigated. For the analysis of this method, we show that this technique is related to similar work
in \cite{bonito2015numerical} which uses a discretization of the so-called Balakrishnan formula for fractional powers of positive operators \cite{MR0115096,MR1192782}.

Before we move along, let us mention some related approaches to the numerical solution to \eqref{eq:pde-problem}. For a more thorough overview,
we refer the reader to \cite{bonito2018numerical,MR4189291,MR4132117,MR4043885}.
Reference \cite{vabishchevich2015numerically} studies the numerical solution to problems with elliptic operators with a fractional power.
In order to apply the method proposed in \cite{vabishchevich2015numerically} to problem \eqref{eq:pde-problem} we first observe that, if we set $A  = -\Delta$, we have that
$A=A^* \geq \delta I$, where $\delta > 0$. The problem at hand is then transformed into a pseudo-parabolic problem
\begin{equation*}
\big(t(A - \delta I)+\delta I\big) \frac{dw}{dt} + s (A -\delta I) w = 0 \, .
\end{equation*}
with $w(0) = \delta^{-s} f$. The solution of this problem is related to our problem by
$u = w(1)$. Reference \cite{MR4122489} provides further insight and an analysis of this method.

Another class of approaches uses the idea of rational approximation. In \cite{aceto2017rational} this is applied to a fractional-in-space reaction-diffusion equation
\begin{equation*}
 \frac{\partial u(x,t)}{\partial t} = - \kappa_{2s} (-\Delta)^s u(x,t) + f(x,t,u) \, \quad x \in \Omega \, , \, t\in (0,T) \, ,
\end{equation*}
where $\kappa_{2s}$ is the diffusion coefficient.
The equation is subject to the boundary condition $u(x,t) |_{x \in \partial \Omega}=0$ and initial condition $u(x,0)=u_0(x)$.
Note that in this work $s$ is restricted to $(1/2,1]$. The first step in this method is to discretize the Laplacian by a finite
difference method over a uniform mesh resulting in the matrix approximation $h^{-2}L$. Then, the scheme applies the fractional power $s$
to the matrix $L$ obtaining
\begin{equation*}
 (-\Delta)^s \approx  \frac{1}{h^{2s}} L^s \approx \frac{1}{h^{2s}} M^{-1} K \, ,
\end{equation*}
where $M$ and $K$ are banded matrices that are found by a rational approximation of $z^{s-1}$ using the Gauss-Jacobi rule applied to the Dunford-Taylor integral
representation of $L^s$ \cite{aceto2017rational}.
A similar approach, but one using a best uniform rational approximation (BURA) is proposed in
\cite{harizanov2018optimal, harizanov2020analysis}. These works start with a finite element discretization of an elliptic operator, $\hat{\mathcal{A}}$,
and consider solving the algebraic problem
$\mathcal{A}^s \mathbf{u} = \mathbf{f}$,
where $\mathcal{A}$ is a rescaling of $\hat{\mathcal{A}}$ so that the spectrum lies in $(0,1]$. The authors of \cite{harizanov2018optimal}
then determine the BURA to $z^{1-s}$ on the interval $(0,1]$ using a modified Remez algorithm.

A scheme involving a quadrature formula for the Dunford-Taylor integral representation
of the fractional Laplacian is developed in \cite{bonito2015numerical}. In particular, it is observed that the Dunford-Taylor integral is equivalent to the Balakrishnan integral
\begin{equation*}
 (-\Delta)^{-s} = \frac{\sin(s\pi)}{\pi} \int_0^\infty t^{-s} (t \mathrm{I} - \Delta)^{-1} dt \, .
\end{equation*}
To approximate the solution $(t_i \mathrm{I} - \Delta)^{-1} f $ is computed via a FEM at the quadrature nodes $t_i$.
Using a change of variables, \cite{bonito2015numerical} shows that the quadrature converges exponentially fast with respect to the number
of quadrature nodes. Subsequent work extended the approach to more general operators in \cite{bonito2017numerical}.
The following error estimate was proved for convex domains $\Omega$ \cite{bonito2015numerical, bonito2018numerical}:
\begin{equation*}
 \| u - u_h^k \|_{L^2(\Omega)} \lesssim h^2 \| f \|_{\mathbb{H}^{2-2s}(\Omega)} \, ,
\end{equation*}
where $u$ is the solution to \eqref{eq:pde-problem} and $u_h^k$ represents its numerical approximation.

Other approaches have been proposed that utilize the Caffarelli-Silvestre extension but with different tactics than those in
\cite{nochetto2015pde, banjai2019tensor} or in this work.
A hybrid FE-spectral method to solve the extended problem 
is proposed in \cite{ainsworth2018frac}.
Their approach uses the first $M$ eigenvalues of the Laplacian. The lower part of the spectrum is approximated by a FEM on $\Omega$
and the upper part of the spectrum is approximated using Weyl's asymptotic \cite{ivrii2016100}. For convenience let us denote by $h$ the mesh size
of the finite element space $\mathbb{V}_h$, defined in $\Omega$, $k$ the polynomial order of the FEM in $\mathbb{V}_h$, $r>-s$ the regularity of $f$,
and $\mathcal{U}_h^M$ the discrete solution. Under the assumption that $M$ is large enough so that $\lambda_M^{-(r+s)/2} \sim h^{\min \{k, r+s\}}$,
the following error estimate is provided
\cite[Theorem 2]{ainsworth2018frac}:
\begin{equation*}
 \| \nabla ( \mathcal{U} - \mathcal{U}_{h}^M )\|_{L^2 ( y^\alpha, \mathcal{C}_{\mathcal{Y}} ) } \\
   \lesssim h^{\min \{ k, r+s \}} \sqrt{ |\log h| } \,\, \| f \|_{\mathbb{H}^{r}(\Omega)} \, .
\end{equation*}
If $M \sim \mathcal{O}(\log^p \text{dim} \mathbb{V}_h)$ for some $p\geq 0$, then
\cite[Equation (28)]{ainsworth2018frac}, for some $q\geq 0$,
\begin{equation*}
  \| \nabla ( \mathcal{U} - \mathcal{U}_h^M )\|_{L^2 ( y^\alpha, \mathcal{C}_{\mathcal{Y}} ) } \\
  \lesssim | \log ( M \cdot \text{dim}\mathbb{V}_h ) |^q \big( M \cdot \text{dim} \mathbb{V}_h \big)^{\frac{-\min \{ k, r+s \}}{d}} \| f \|_{\mathbb{H}^{r}(\Omega)} \, .
\end{equation*}

Reference \cite{chen2020efficient}
uses generalized Laguerre functions as basis for the extended dimension. Since
the weight $y^\alpha$ appears in the problem, the use of Laguerre functions is natural. The scheme is called an enriched Laguerre spectral method:
the basis in the extended dimension is supplemented with additional functions to compensate for the singular/degenerate weight.
Denoting by $N$ the number of generalized Laguerre basis functions and by $k$ the number of additional functions used to enrich the Laguerre basis, the following error estimate is proved in \cite{chen2020efficient}:
\begin{equation*}
  \begin{aligned}
 \| \nabla ( \mathcal{U} - \mathcal{U}_h )\|_{L^2 ( y^\alpha, \mathcal{C}_{\mathcal{Y}} ) } &\lesssim N^{-\frac{m}2} \| f \|_{\mathbb{H}^{m/2-s}(\Omega)} + N^{1-\frac{m}2} \| f\|_{\mathbb{H}^{(m-1)/2-s}(\Omega)} \\
 &+ h^r \| f \|_{\mathbb{H}^{r+1/2-s}(\Omega)} + h^{r-1}\| f \|_{\mathbb{H}^{r-1-s}(\Omega)} \, ,
 \end{aligned}
\end{equation*}
where $r>-s$, $f \in \mathbb{H}^{l-s}(\Omega)$ with $l=\max\{\tfrac{m}2,r+\tfrac12\}$, and $m=2k+1+[2s]$.

Finally, \cite{hofreither2020unified} showed that all the schemes in
\cite{vabishchevich2015numerically,aceto2017rational,harizanov2018optimal,bonito2015numerical,nochetto2015pde,banjai2019tensor}
can be understood as methods based on rational approximation of a univariate function.

We now briefly summarize the convergence results from the previous work in \cite{nochetto2015pde,banjai2019tensor} that we build upon here.
As before, we will denote the discrete solution by $\mathcal{U}_h$. Recall, $\mathbb{V}_h$ denotes the finite element space. By $\mathcal{S}_{\mathcal{Y}}$
we denote the finite element space in the extended direction. For the following estimates, $\mathbb{V}_h$ will consist of piecewise linear functions
on a conforming and quasi-uniform mesh of quadrilaterals on $\Omega$.
If $\mathcal{S}_{\mathcal{Y}}$ is taken to be piecewise linear functions on quasi-uniform intervals, the following error estimate holds
\cite[Theorem 5.1]{nochetto2015pde}:
\begin{equation*}
 \label{eq-error-uniform}
 \| \nabla ( \mathcal{U} - \mathcal{U}_h )\|_{L^2 ( y^\alpha, \mathcal{C}_{\mathcal{Y}} ) } \\
   \lesssim h^{s-\epsilon} \| f \|_{\mathbb{H}^{1-s}(\Omega)}
   \sim \big( \text{dim}\mathbb{V}_h \cdot \text{dim} \mathcal{S}_{\mathcal{Y}} \big)^{- \frac{s-\epsilon}{d+1} } \| f \|_{\mathbb{H}^{1-s}(\Omega)} \, .
\end{equation*}
If instead $\mathcal{S}_{\mathcal{Y}}$ consists of piecewise linear functions on intervals that are graded towards $y=0$, the error estimate
can be improved as shown in \cite[Equation (5.11)]{nochetto2015pde} to
\begin{equation*}
 \label{eq:error-graded}
 \| \nabla ( \mathcal{U} - \mathcal{U}_h )\|_{L^2 ( y^\alpha, \mathcal{C}_{\mathcal{Y}} ) } \\
   \lesssim | \log ( \text{dim}\mathbb{V}_h \cdot \text{dim} \mathcal{S}_{\mathcal{Y}} ) |^s
            \big( \text{dim}\mathbb{V}_h \cdot \text{dim} \mathcal{S}_{\mathcal{Y}} \big)^{-\frac{1}{d+1}}
            \| f \|_{\mathbb{H}^{1-s}(\Omega)} \, .
\end{equation*}
Lastly, in \cite{banjai2019tensor}, $\mathcal{S}_{\mathcal{Y}}$ uses \textit{hp}-FEM on intervals with a geometric grading towards $y=0$. Choosing $q>0$ the error estimate is given from \cite[Equation (5.57)]{banjai2019tensor} as
\begin{equation*}
 \label{eq:error-hp}
 \| \nabla ( \mathcal{U} - \mathcal{U}_h )\|_{L^2 ( y^\alpha, \mathcal{C}_{\mathcal{Y}} ) }
   \lesssim | \log ( \text{dim}\mathbb{V}_h ) |^q \big( \text{dim} \mathbb{V}_h \big)^{-\frac{1}{d}} \| f \|_{\mathbb{H}^{1-s}(\Omega)} \, .
\end{equation*}

Having reviewed the existing methods and their analysis we can proceed with the development of our semi-analytic technique.
Our presentation shall be organized as follows. In \Cref{sec:prelim}, we discuss the notation and preliminary results needed to study
the problem. In \Cref{sec:diagonalization}, we discuss the diagonalization technique and the analytical solution
to the eigenvalue problem. \Cref{sec:quadrature-error} contains the analysis of our method. We draw a connection between
the exact diagonalization scheme and quadrature of the Balakrishnan formula for the inverse of the fractional Laplacian operator
along with error estimates. Lastly, in \Cref{sec:results}, numerical results are presented to demonstrate the convergence of the scheme
as well as the parallel performance of the algorithm.

\section{Preliminaries and Notation}
\label{sec:prelim}

Throughout our work we assume that $\Omega$ is a convex, open, bounded, and connected subset of $\mathbb{R}^d$, $d \geq 1$.
We will assume that $\partial \Omega$ is polyhedral. We define the semi-infinite cylinder as $\mathcal{C} = \Omega \times (0, \infty)$,
along with its lateral boundary $\partial_L \mathcal{C} = \partial \Omega \times [0, \infty)$. We similarly define, for any $\mathcal{Y} > 0$,
the truncated cylinder $\mathcal{C}_{\mathcal{Y}} = \Omega \times (0 , \mathcal{Y})$ and its corresponding lateral boundary
$ \partial_L \mathcal{C}_{\mathcal{Y}} = \partial \Omega \times [0,\mathcal{Y}] $.
It will also be prudent to distinguish between points in $\mathbb{R}^d$ and those in $\mathbb{R}^{d+1}_+$.
We set $\mathbf{x}=(x,y)$ where $x \in \mathbb{R}^d$ and $y \in (0,\infty)$.
We use $a \lesssim b$ to indicate that $a \leq C b$, where the constant $C$ is independent of either $a$ or $b$ and does not
depend on the discretization scheme. By $a \sim b$, we mean that $a \lesssim b \lesssim a$, where the constants on either end
of the inequalities are generally different from one another.
We use standard Sobolev spaces $H^1(\Omega)$, $H_0^1(\Omega)$, and $H^\sigma (\Omega)$, with $\sigma \in \mathbb{R}$
\cite{adams2003sobolev, di2012hitchhikers}.

\subsection{The Spectral Fractional Laplacian}
\label{subsec:frac-laplace-op}
In this work we use the spectral interpretation of $(-\Delta)^s$ which we briefly summarize below.
Define
\[
 (-\Delta):L^2(\Omega) \to L^2(\Omega) \, , \qquad  \text{Dom } (-\Delta)= \{ w \in H_0^1(\Omega): \Delta w \in L^2(\Omega) \}.
\]
This operator is positive, unbounded, closed, and possesses a compact inverse. Hence, the spectrum of $(-\Delta)$ is discrete, positive,
and accumulates at infinity.

In other words, we have the existence of $\{ (\lambda_k , \phi_k) \}_{k \geq 1} \in \mathbb{R}_+ \times H_0^1(\Omega)\setminus \{0\}$, where
\begin{equation*}
 \label{eq:laplace-eigenpair}
 -\Delta \phi_k = \lambda_k \phi_k \,  \quad \text{in } \Omega \quad \text{and}  \quad \phi_k = 0 \,  \quad \text{on } \partial \Omega \, .
\end{equation*}
In addition, $\lambda_1 = \min\{\lambda_k \}_{k \geq 1}$ is simple and $\lambda_k \to \infty$ as $k \to \infty$. Moreover, $\{ \phi_k \}_{k\geq 1}$ forms an orthonormal basis of $L^2(\Omega)$ \cite[Section 6.5, Theorem 1]{evans2022partial}. Finally, $\{ \phi_k \}_{k \geq 1}$ forms an orthogonal basis of $H_0^1(\Omega)$ with
$\| \nabla \phi_k \|_{L^2(\Omega)}^2 = \lambda_k$.

We can now define fractional powers of the Dirichlet Laplacian using this spectral decomposition. For any $w \in C_0^\infty(\Omega)$ we set $w_k = \int_\Omega w \phi_k \, dx$ and
\begin{equation}
 \label{eq:spectral-fractional-laplacian}
 (-\Delta)^s w = \sum_{k=1}^\infty w_k \lambda_k^s \phi_k \, .
\end{equation}

Suppose now that we wish to solve \eqref{eq:pde-problem} and $f = \sum_{k\geq 1} f_k \phi_k$. Owing to the orthonormality of
$\{\phi_k\}$, if we define $u_k = \lambda_k^{-s} f_k$ for all $k \in \mathbb{N}$, then the function $u = \sum_{k \geq 1} u_k \phi_k$ solves \eqref{eq:pde-problem}, provided all these series converge in a suitable sense. This is made rigorous by introducing, for $r \geq 0$, the following Hilbert space:
\begin{equation*}
 \label{eq:hs-bb}
 \mathbb{H}^r(\Omega) = \left\{ w = \sum_{k=1}^\infty w_k \phi_k \in L^2(\Omega) : \| w \|_{\mathbb{H}^r(\Omega)}^2
   \coloneqq \sum_{k=1}^\infty \lambda_k^r |w_k|^2 < \infty \right\}.
\end{equation*}
For $r<0$ the space $\mathbb{H}^r(\Omega)$ is the dual of $\mathbb{H}^{-r}(\Omega)$. Further characterizations of these spaces can be seen in \cite[Equation (2.13)]{nochetto2015pde}.

We now can give a meaning to \eqref{eq:pde-problem}. Namely, for $s \in (0,1)$, the operator $(-\Delta)^s : \mathbb{H}^s(\Omega) \to \mathbb{H}^{-s}(\Omega)$ is an isometry. For this reason, if $f = \sum_{k \geq 1} f_k \phi_k \in \mathbb{H}^{-s}(\Omega)$, the solution to \eqref{eq:pde-problem} satisfies $\| u \|_{\mathbb{H}^s(\Omega)} = \| f \|_{\mathbb{H}^{-s}(\Omega)}$ and is given by
\begin{equation}
\label{eq:SolFracLapSeriesSense}
  u = \sum_{k\geq1} \lambda_k^{-s} f_k \phi_k \in \mathbb{H}^s(\Omega) .
\end{equation}

\subsection{The Caffarelli-Silvestre Extension}
\label{subsec:c-s-ext}
The Caffarelli-Silvestre extension \cite{caffarelli2007extension} necessitates the handling of degenerate/singular elliptic
equations on the space $\mathbb{R}^{d+1}_+=\mathbb{R}^d\times \mathbb{R}_+$, or in the case of a bounded domain, the extension
\cite{cabre2010positive,stinga2010extension} requires handling such equations on $\mathcal{C}$.
To study such equations weighted Sobolev spaces \cite{kufner1984define} are used with the weight $y^\alpha$, where
$\alpha$ is defined in \eqref{eq:define-alpha}.

We start with some domain $\mathcal{D} \subset \mathbb{R}^{d+1}_+$ and define $L^2 (y^\alpha, \mathcal{D} )$ as the space
containing all measurable functions $w$ defined on $\mathcal{D}$ such that
\begin{equation*}
 \label{eq:weighted-L2}
 \| w \|_{L^2(y^\alpha, \mathcal{D})}^2 \coloneqq \int_{\mathcal{D}} |y|^\alpha w(\mathbf{x})^2 \, d \mathbf{x} < \infty \, .
\end{equation*}
Using this definition, we can similarly define the weighted Sobolev space
\begin{equation*}
 \label{eq:weighted-sobolev}
 H^1 ( y^\alpha, \mathcal{D} ) = \left\{ w \in L^2(y^\alpha, \mathcal{D}) : |\nabla w| \in L^2(y^\alpha, \mathcal{D} ) \right\} \, ,
\end{equation*}
where $ \nabla w$ are distributional derivatives. The space is equipped with the norm:
\begin{equation*}
 \label{eq:weighted-sobolev-norm}
 \| w \|_{H^1(y^\alpha, \mathcal{D})}^2 = \| w \|_{L^2( y^\alpha, \mathcal{D})}^2 + \| \nabla w \|_{L^2( y^\alpha, \mathcal{D})}^2.
\end{equation*}
The Sobolev space, defined on $\mathcal{C}$, of functions that vanish at the lateral boundary is
\begin{equation*}
 \label{eq:h10-circle}
 \mathring{H}_L^1 (y^\alpha, \mathcal{C}) = \left\{ w \in H^1(y^\alpha, \mathcal{C}) : w|_{\partial_L \mathcal{C}} = 0 \right\} \, .
\end{equation*}
We note that this space has a weighted Poincar\'{e} inequality \cite[(2.21)]{banjai2019tensor} and as a result the
$\mathring{H}_L^1(y^\alpha,\mathcal{C})$ semi-norm is an equivalent norm on this space.
Finally, we comment that functions in $\mathring{H}_L^1 (y^\alpha, \mathcal{C})$ have well defined traces in $\Omega \times \{0\}$. By this we mean that there is a continuous
and surjective mapping $\text{tr } : \mathring{H}_L^1 (y^\alpha, \mathcal{C}) \to \mathbb{H}^s(\Omega)$ such that, according to \cite[Proposition 2.5]{nochetto2015pde}, for every smooth $w \in \mathring{H}_L^1 (y^\alpha, \mathcal{C})$
\[
  \text{tr } w(x) = w(x,0), \qquad \forall x \in \Omega \, .
\]

We define the, bounded and coercive, bilinear form
\begin{equation*}
 \label{eq:bilinear-form-cylinder}
a_{\mathcal{C}}: \mathring{H}_L^1(y^\alpha,\mathcal{C}) \times \mathring{H}_L^1(y^\alpha,\mathcal{C}) \to \mathbb{R}
\qquad 
 a_{\mathcal{C}} (v,w) = \int_{\mathcal{C}} y^\alpha \nabla v \cdot \nabla w  \, d \mathbf{x} \, .
\end{equation*}
The weak formulation of \eqref{eq:extended-problem} is:
find $\mathcal{U}\in \mathring{H}_L^1(y^\alpha,\mathcal{C})$ such that
\begin{equation}
 \label{eq:cs-weak-form}
 a_{\mathcal{C}}( \mathcal{U}, v ) = d_s \langle f , \text{tr } v \rangle \, , \quad \forall \, v \in \mathring{H}_L^1(y^\alpha, \mathcal{C}) \, ,
\end{equation}
where $\langle \cdot, \cdot \rangle$ is the duality pairing between $\mathbb{H}^s(\Omega)$ and $\mathbb{H}^{-s}(\Omega)$.

The main result of Caffarelli and Silvestre can now be given.

\begin{theorem}[extension]
Let $f \in \mathbb{H}^{-s}(\Omega)$ and assume that $\mathcal{U} \in \mathring{H}_L^1(y^\alpha, \mathcal{C})$ solves
\eqref{eq:cs-weak-form}. Then $u = \text{tr } \mathcal{U} \in \mathbb{H}^s(\Omega)$ solves \eqref{eq:pde-problem} in the sense that it satisfies \eqref{eq:SolFracLapSeriesSense}.
\end{theorem}
\begin{proof}
See \cite{caffarelli2007extension} for the case $\Omega = \mathbb{R}^d$, and \cite{cabre2010positive,stinga2010extension} for bounded domains.
\end{proof}

\subsection{Truncation}

As a first step towards the numerical approximation of the solution to \eqref{eq:pde-problem} via \eqref{eq:cs-weak-form}
we truncate the infinite cylinder. Define, for $\mathcal{Y}>0$,
\begin{equation*}
 \label{eq:weighted-truncated-space}
 \mathring{H}_L^1(y^\alpha, \mathcal{C}_{\mathcal{Y}} ) = \left\{ w \in H^1(y^\alpha, \mathcal{C}_{\mathcal{Y}}) :
   w |_{\partial_L \mathcal{C}_{\mathcal{Y} } \, \cup \, (\Omega \times \{\mathcal{Y}\} )} = 0 \right\} \, .
\end{equation*}
We introduce the following problem: find $\mathcal{U}_{\mathcal{Y}} \in \mathring{H}_L^1(y^\alpha,\mathcal{C}_{\mathcal{Y}} )$ such that
\begin{equation}
 \label{eq:cs-weak-form-trunc}
 a_{\mathcal{C}_{\mathcal{Y}} }( \mathcal{U}_{\mathcal{Y}}, v ) = d_s \langle f , \text{tr } v \rangle \, , \quad \forall \, v \in \mathring{H}_L^1(y^\alpha, \mathcal{C}_{\mathcal{Y}} ) \, ,
\end{equation}
where the bilinear form $a_{\mathcal{C}_{\mathcal{Y}} }$ is defined in an analogous manner to $a_{\mathcal{C}}$.

The truncation error is quantified below.

\begin{proposition}[truncation]
Let $\mathcal{Y}\geq 1$. If $\mathcal{U} \in \mathring{H}_L^1(y^\alpha, \mathcal{C} )$ solves \eqref{eq:cs-weak-form}, then
\[
  \int_{\mathcal{Y}}^\infty\int_\Omega y^\alpha |\nabla \mathcal{U}|^2 \ d \mathbf{x} \lesssim e^{- \sqrt{\lambda_1} \mathcal{Y}} \| f \|_{\mathbb{H}^{-s}(\Omega)}^2.
\]
Consequently, if $\mathcal{U}_{\mathcal{Y}} \in \mathring{H}_L^1(y^\alpha,\mathcal{C}_{\mathcal{Y}} )$ solves \eqref{eq:cs-weak-form-trunc}, we have
\begin{equation*}
 \label{eq:exp-conv-trunc}
 \| u - \text{tr } \mathcal{U}_{\mathcal{Y}} \|_{\mathbb{H}^s(\Omega)} = \| \nabla ( \mathcal{U} - \mathcal{U}_{\mathcal{Y}} )\|_{L^2(y^\alpha, \mathcal{C} )} \lesssim e^{-\sqrt{\lambda_1} \, \mathcal{Y}/4} \| f \|_{\mathbb{H}^{-s}(\Omega)} \, .
\end{equation*}
\end{proposition}
\begin{proof}
  See \cite[Theorem 3.5]{nochetto2015pde}.
\end{proof}

\subsection{Tensorial Discretization}
\label{subsec:tensor-fem}

Let us recall the following a general result.

\begin{theorem}[tensor products]
\label{thm:tensor-poducts}
Let $(M_i,\mu_i)$, $i=1,2$, be measure spaces. Then we have the isomorphism
\[
  L^2(M_1,\mu_1) \otimes L^2(M_2,\mu_2) \eqsim L^2( M_1\times M_2, \mu_1 \otimes \mu_2 ) \, ,
\]
provided that all spaces are separable. In addition, if $\{\phi_k^{(i)}\}_{k\geq 1}$ is an orthonormal basis of $L^2(M_i,\mu_i)$, then $\{\phi_k^{(1)}\phi_m^{(2)}\}_{k,m \geq1}$ is an orthonormal basis of their tensor product.
\end{theorem}
\begin{proof}
  See \cite[Theorem II.10]{reed1981functional} and \cite[Proposition II.4.2]{reed1981functional}.
\end{proof}

The following corollary motivates, as in \cite{nochetto2015pde,banjai2019tensor}, the use of a tensorial discretization.

\begin{corollary}[tensor products]
\label{cor:tensor-products}
Let 
\begin{equation*}
 \label{eq:h1-L-Y}
 H_{\mathcal{Y}}^1 (y^\alpha, (0,\mathcal{Y}) ) = \{ w \in L^2 (y^\alpha , (0,\mathcal{Y}) ) : w' \in L^2 (y^\alpha , (0,\mathcal{Y}) ) \, , \,
   w(\mathcal{Y}) = 0 \} \, .
\end{equation*}
Then we have
\begin{equation}
 \label{eq:equiv-spaces-tensor}
 \mathring{H}_L^1(y^\alpha, \mathcal{C}_{\mathcal{Y}} ) = H_0^1(\Omega) \otimes H_{\mathcal{Y}}^1(y^\alpha, (0,\mathcal{Y}) ) \, .
\end{equation}
\end{corollary}
\begin{proof}
Let $\{ \phi_j \}_{j \geq 1} \subset H_0^1(\Omega)$ and
$\{ \psi_k \}_{k \geq 1} \subset H_{\mathcal{Y}}^1 (y^\alpha,(0,\mathcal{Y}))$ be orthonormal bases which, in addition, are orthogonal in $L^2(\Omega)$ and $L^2(y^\alpha,(0,\mathcal{Y}))$, respectively.
We can, for instance, choose $\{\phi_j\}_{j\geq1}$ to be the eigenfunctions of the Dirichlet Laplacian, as described in \Cref{subsec:frac-laplace-op}. An example of the family $\{\psi_k \}_{k \geq 1}$ will be given below in \Cref{thm:eigenpairsyalpha}.

Clearly, $\phi_j \psi_k \in \mathring{H}_L^1 (y^\alpha, \mathcal{C}_{\mathcal{Y}})$. Observe, in addition, that
\begin{align*}
 (\phi_j \psi_k , \phi_s \psi_t)_{\mathring{H}_L^1(y^\alpha, \mathcal{C}_{\mathcal{Y}})} &= \int_0^{\mathcal{Y}} \int_\Omega \big( [\nabla \phi_j(x) \cdot \nabla \phi_s(x)]
     \psi_k(y) \psi_t(y) \\
   &  \quad \qquad \qquad + \phi_j(x) \phi_s(x) \psi_k'(y) \psi_t'(y) \big) \, dx \, dy \\
   &= \delta_{js} \int_0^{\mathcal{Y}} y^\alpha \psi_k(y) \psi_t(y) \, dy + \delta_{kt} \int_\Omega \phi_j(x) \phi_s(x) \, dx \\
   &= C_{jkst} \delta_{js} \delta_{kt} \, .
\end{align*}
Thus, the family $\{ \phi_j \psi_k \}_{j,k \geq 1}$ is linearly independent. Using Gram-Schmidt, this linearly independent set contains an orthonormal subset which, to simplify the notation, we do not relabel.

Next we show that $\{\phi_j \psi_k\}_{j,k\geq 1}$ is complete in $\mathring{H}_L^1 (y^\alpha, \mathcal{C}_{\mathcal{Y}})$. To see this we assume that
$W \in \mathring{H}_L^1 (y^\alpha, \mathcal{C}_{\mathcal{Y}})$ is such that
\[
 \int_0^{\mathcal{Y}} y^\alpha \int_\Omega \left( [\nabla_x W(x,y) \cdot \nabla \phi_j(x)]\psi_k(y) + \partial_y W(x,y) \phi_j(x)\psi_k'(y) \right) \, dx \, dy = 0 \, ,
\]
for all $j,k \in \mathbb{N}$. This implies that
\begin{align*}
 0 &= \int_0^{\mathcal{Y}} y^\alpha  \psi_k(y) \left( \int_\Omega \nabla_x W(x,y) \cdot \nabla_x \phi_j(x) \, dx \right)\, dy \, + \\
   & \quad \int_\Omega \phi_j(x) \left( \int_0^{\mathcal{Y}} y^\alpha \partial_y W(x,y) \psi_k'(y) \, dy  \right) \, dx \, .
\end{align*}
But, because $\phi_j$ and $\psi_k$ are linearly independent, this is only possible if
\[
 \int_\Omega \nabla_x W(x,y) \cdot \nabla_x \phi_j(x) \, dx = 0 \, ,
\]
for all $j \in \mathbb{N}$ and a.e.~$y \in (0, \mathcal{Y})$; and
\[
 \int_0^{\mathcal{Y}} y^\alpha \partial_y W(x,y) \psi_k'(y) \, dy = 0 \, ,
\]
for all $k \in \mathbb{N}$ and a.e.~$x \in \Omega$. Let now, for each $j \in \mathbb{N}$, $S_j \subset (0,\mathcal{Y})$ be the set where the first condition is false. Since $S_j$ is null, i.e.,
\[  \int_{S_j} y^\alpha \, dy = 0 \, , \]
we have that $S = \bigcup_{j=1}^\infty S_j$ is also a null set, as it is a countable union of null sets.

Similarly, for $k \in \mathbb{N}$ we let $T_k \subset \Omega$ is the where the second condition fails. We again have that
\[  \int_{T_k} dx = 0 \, , \]
and that $T = \bigcup_{k=1}^\infty T_k$ is another null set.

The previous observations show that
\[
  A(x,y) = |\nabla_x W(x,y)| + |\partial_y W(x,y)|=0, \quad \forall (x,y) \in \mathcal{C}_{\mathcal{Y}} \setminus (T \times S).
\]
However, $T \times S$ is a null set and therefore, $A(x,y) = 0$ almost everywhere in $\mathcal{C}_{\mathcal{Y}}$. As a consequence, $W=0$ a.e. in $\mathcal{C}_{\mathcal{Y}}$.

In conclusion, we have shown that $\{ \phi_j \psi_k\}_{j,k\geq 1}$ is a complete orthonormal basis in $\mathring{H}_L^1(y^\alpha, \mathcal{C}_{\mathcal{Y}})$.

This shows that the mapping defined by
\begin{align*}
  \mathcal{W} : H_0^1(\Omega) \otimes H_{\mathcal{Y}}^1 ( y^\alpha, (0, \mathcal{Y}) ) &\to \mathring{H}_L^1(y^\alpha , \mathcal{C}_{\mathcal{Y}}) \\
  \phi_j \otimes \psi_k &\mapsto \phi_j \psi_k \, ,
\end{align*}
and extended by linearity, is the requisite isomorphism.
\end{proof}

We now proceed to approximate the solution of \eqref{eq:cs-weak-form-trunc} via a Galerkin technique.
We approximate $H_0^1(\Omega)$ by a finite element space consisting of piecewise linear and continuous functions that vanish on the boundary
\cite{ern2004theory}, which we denote $\mathbb{V}_h$. Usually this will be constructed on the basis of a quasi-uniform triangulation of $\Omega$
of size $h>0$ so that
$
  \dim \mathbb{V}_h \sim h^{-d} \, .
$
For the extended dimension, at this stage, we merely introduce a finite dimensional subspace $\mathcal{S}_{\mathcal{Y}}$ of $H^1_L(y^\alpha,(0,\mathcal{Y}))$.
We assume, for $\mathcal{K} \in \mathbb{N}$, that
$
  \dim \mathcal{S}_{\mathcal{Y}} \sim \mathcal{K} \, .
$
Clearly, $\dim \mathbb{V}_h \otimes \mathcal{S}_{\mathcal{Y}} \sim \mathcal{K} h^{-d}$ and, as shown in \eqref{eq:equiv-spaces-tensor},
\begin{equation*}
 \label{eq:tensor-fem-disc}
 \mathbb{V}_h \otimes \mathcal{S}_{\mathcal{Y}} \subset H_0^1(\Omega) \otimes H_{\mathcal{Y}}^1(y^\alpha, (0,\mathcal{Y}) ) =
 \mathring{H}_L^1(y^\alpha, \mathcal{C}_{\mathcal{Y}} ) \, .
\end{equation*}

The Galerkin approximation of the solution to \eqref{eq:cs-weak-form-trunc} is then denoted by $\mathcal{U}_{h,\mathcal{Y}}^{\mathcal{K}} \in \mathbb{V}_h \otimes \mathcal{S}_{\mathcal{Y}}$. From its definition, the best approximation result immediately follow:
\begin{equation*}\label{eq:CeaForAnyDiscretization}
    \| \text{tr } (\mathcal{U}_{\mathcal{Y}} - \mathcal{U}_{h,\mathcal{Y}}^{\mathcal{K}}) \|_{\mathbb{H}^s(\Omega)} = \| \nabla( \mathcal{U}_{\mathcal{Y}} - \mathcal{U}_{h,\mathcal{Y}}^{\mathcal{K}} ) \|_{L^2(y^\alpha,\mathcal{C}_{\mathcal{Y}}) }
    = \inf_{W \in \mathbb{V}_h \otimes \mathcal{S}_{\mathcal{Y}}} \| \nabla( \mathcal{U}_{\mathcal{Y}} - W ) \|_{L^2(y^\alpha,\mathcal{C}_{\mathcal{Y}}) } \, .
\end{equation*}
Some choices of $\mathcal{S}_{\mathcal{Y}}$ that have been used in the literature, and the corresponding error estimates, are detailed in \Cref{sec:intro}.

\section{Diagonalization}
\label{sec:diagonalization}

Let us now detail the diagonalization approach originally suggested in \cite{banjai2019tensor}, and whose extension shall be the basis of our approach.

\subsection{Background}
The tensorial discretization leads to a representation of the Galerkin solution to \eqref{eq:cs-weak-form-trunc}: For $(x,y) \in \mathcal{C}_{\mathcal{Y}}$,
\begin{equation}
 \label{eq:sol-rep}
 \mathcal{U}_{h,\mathcal{Y}}^{\mathcal{K}} (x,y) = \sum_{k=1}^{\mathcal{K}} U_k(x) v_k(y) \,\, , \,\,
 \qquad
 u_{h,\mathcal{Y}}^{\mathcal{K}}(x) = \mathcal{U}_{h,\mathcal{Y}}^{\mathcal{K}} (x,0) = \sum_{k=1}^{\mathcal{K}} U_k(x) v_k(0) \, ,
\end{equation}
where $U_k \in \mathbb{V}_h $ and $ v_k \in \mathcal{S}_{\mathcal{Y}} $. As a means to reduce computational complexity, a diagonalization
technique is applied in \cite{banjai2019tensor}, where $\mathcal{S}_{\mathcal{Y}}$ is an \textit{hp}-FE space over a geometrically graded mesh.
Then, the elements $v_k$ in \eqref{eq:sol-rep} are chosen as the solutions to the following discrete eigenvalue problem:
find $(v, \mu) \in \mathcal{S}_{\mathcal{Y}} \setminus \{ 0 \} \times \mathbb{R}$ such that
\begin{equation}
 \label{eq:eigenval-problem}
 \int_0^{\mathcal{Y}} y^\alpha v'(y) w'(y) \, dy = \mu \int_0^{\mathcal{Y}} y^\alpha v(y) w(y) \, dy \,\, , \,\, \forall \, w\in \mathcal{S}_{\mathcal{Y}} \, .
\end{equation}

A main drawback to this approach is that computing eigenvalues is inherently numerically unstable; see \cite{zhang2015many}. Some attempts at addressing this,
by suitably choosing the parameters in the \textit{hp}-FEM are discussed in the Section on Numerical Experiments and the Conclusions of \cite{MR4530200}. This,
however, is an essential difficulty. The eigenvalue problem \eqref{eq:eigenval-problem} is very ill-conditioned.

Ignoring conditioning issues,
let us explain what is the advantage of this diagonalization. We begin by observing that, as shown in \cite{banjai2019tensor}, the solutions to
\eqref{eq:eigenval-problem} form an orthonormal basis of the FE space $\mathcal{S}_{\mathcal{Y}}$ which, in addition, satisfy
\[
  \int_0^{\mathcal{Y}} y^\alpha v_k'(y) v_i'(y) dy = \mu_k \delta_{ki} \, .
\]
We may then write \eqref{eq:sol-rep} using as $v_k$ the solutions to \eqref{eq:eigenval-problem}. Choosing, for our test function, $\zeta = Vv_i$,
where $V\in \mathbb{V}_h$, we obtain
\begin{align*}
  &\int_{\Omega} \int_0^{\mathcal{Y}} y^\alpha \nabla \mathcal{U}_{h,\mathcal{Y}}^{\mathcal{K}} (x,y) \cdot \nabla \big( V(x) v_i(y) \big) \, dx \, dy \\
  &= \int_\Omega \int_0^{\mathcal{Y}} y^\alpha \nabla \left( \sum_{k=1}^{\mathcal{K}} U_k(x) v_k(y) \right) \cdot \nabla \big( V(x) v_i(y) \big) \, dx \, dy  \\
  &= \sum_{k=1}^{\mathcal{K}} \int_\Omega \left[ \nabla U_k(x) \cdot \nabla V(x) \int_0^{\mathcal{Y}} y^\alpha v_k(y) v_i(y)  dy
  +  U_k(x) V(x) \int_0^{\mathcal{Y}} y^\alpha v_k'(y) v_i'(y)  dy \right]  dx \\
  &= \int_\Omega \nabla U_k(x) \cdot \nabla V(x) \, dx + \mu_k \int_\Omega U_k(x) V(x) \, dx .
\end{align*}
This, for $k = 1, \ldots, \mathcal{K}$, implies the functions $U_k \in \mathbb{V}_h$ solve
\begin{equation}
 \int_\Omega \nabla U_k(x) \cdot \nabla V(x) \, dx + \mu_k \int_\Omega U_k(x) V(x) \, dx = d_s v_k(0) \langle f, V\rangle \, .
\label{eq:FindTheUks}
\end{equation}
The efficiency of the diagonalization technique follows from this decoupling since each of the $\mathcal{K}$ problems now only involves solving over $\Omega$
and they are independent of each other, meaning they may be solved in parallel.

\subsection{A Semi-Analytic Approach}

To avoid numerically solving the eigenvalue problem \eqref{eq:eigenval-problem} we propose an exact diagonalization technique.
We will solve \eqref{eq:eigenval-problem} exactly. We first write the following problem. Find pairs $(\mu,\psi\neq 0)$ such that:
\begin{equation}
  \label{eq:eigen-bvp}
   - (y^\alpha \psi'(y))' = \mu y^\alpha  \psi(y) \  \text{in  } (0,\mathcal{Y}) \, , \quad
   -y^\alpha \psi'(y) \to 0 \text{ as } y \downarrow 0 \, , \text{ and }
   \psi(\mathcal{Y}) = 0 \, .
\end{equation}

\begin{theorem}[eigenpairs]
\label{thm:eigenpairsyalpha}
There is a countable number $\{(\mu_k,\psi_k)\}_{k=1}^\infty \subset \mathbb{R}_+ \times H^1_{\mathcal{Y}}(y^\alpha,(0,\mathcal{Y}))$ of solutions to \eqref{eq:eigen-bvp}. The eigenvalues may be numbered so that 
\begin{equation*}
 \label{eq:sl-eigenvalues}
 0 < \mu_1 < \mu_2 \leq \cdots \leq \mu_n \leq \cdots \,, \qquad \qquad \mu_n \to \infty, \ n \to \infty \, .
\end{equation*}
The functions $\{ \psi_k \}_{k\geq 1}$ are an orthonormal basis of $L^2(y^\alpha, (0,\mathcal{Y}))$. As a consequence,
\begin{equation}
 \label{eq:orthogonal-eigen}
 \int_0^{\mathcal{Y}} y^\alpha \psi_j'(y) \psi_k'(y)\, dy = \mu_j \delta_{jk} \, .
\end{equation}
Finally, $\{ \psi_k \}_{k \geq 1}$ is an orthogonal basis of $H_{\mathcal{Y}}^1 (y^\alpha, (0,\mathcal{Y}))$.
\end{theorem}
\begin{proof}
Notice that this ODE with the given boundary conditions is a singular Sturm-Liouville problem.
Positivity of the first eigenvalue $\mu_1$ is a result of minimizing the Rayleigh quotient over trial functions subject to the boundary conditions
in \eqref{eq:eigen-bvp}. A proof of this can be found in \cite[Theorem 1, Section 11.1]{strauss2007partial} and requires only accounting for the
boundary conditions. The remaining eigenvalues can be found in an analogous manner to \cite[Theorem 2, Section 11.1]{strauss2007partial},
which establishes the orthogonality of the eigenfunctions with respect to $y^\alpha$ as a constraint in minimizing the Rayleigh quotient.
The eigenfunctions can be seen to be complete in the weighted $L^2$ sense as in \cite[Theorem 2, Section 11.3]{strauss2007partial}. Equation
\eqref{eq:orthogonal-eigen} can be seen by integrating the ODE against another eigenfunction
\[
  \int_0^{\mathcal{Y}} y^\alpha \psi_j'(y) \psi_k'(y)\, dy = \int_0^{\mathcal{Y}} -(y^\alpha \psi_j'(y))'\psi_k(y) \, dy = \mu_j \int_0^{\mathcal{Y}} y^\alpha \psi_j (y) \psi_k(y) \, dy = \mu_j \delta_{jk} \, ,
\]
showing that $\{ \psi_k \}_{k \geq 1}$ is an orthogonal basis of $H_{\mathcal{Y}}^1( y^\alpha, (0,\mathcal{Y}))$.
\end{proof}

It is now our intention to find the exact solutions to this problem. Expanding the ordinary differential equation we arrive at
\begin{equation}
 \label{eq:eigen-ode}
 \psi''(y) + \alpha y^{-1} \psi'(y) + \mu \psi(y) = 0 \,\, , \,\, y \in (0,\mathcal{Y}) \, .
\end{equation}
The solutions to this equation behave differently depending on the value of $s$.

\begin{lemma}[eigenpairs for $s = \tfrac12$]
Let $s = \tfrac12$. The solutions to \eqref{eq:eigen-bvp} are given by
\begin{equation}
\label{eq:evs-one-half}
\psi_k(y) = \sqrt{\frac{2}{\mathcal{Y}}} \cos (\mu_k y) \quad \text{and} \quad \mu_k = \left( \frac{ (k-\tfrac12) \pi}{\mathcal{Y}} \right)^2 \, , \quad \forall k\in\mathbb{N} \, .
\end{equation}
\end{lemma}
\begin{proof}
Since $\alpha = 0$, equation \eqref{eq:eigen-ode} becomes a linear second order ODE with constant coefficients. Its solution is carried out by elementary means.
\end{proof}

On the other hand, the case of $s\neq \tfrac12$ requires some work.

\begin{lemma}[eigenpairs for $s \neq \tfrac12$]
Let $s \neq \tfrac12$. The solutions to \eqref{eq:eigen-bvp} are given by
\begin{equation*}
 \label{eq:efs-other}
 \psi_k(y) = \left( \frac{\sqrt{2}}{ \mu_k^{s/2} \mathcal{Y} J_{1-s}(\eta_k) } \right) (y \sqrt{\mu_k})^s J_{-s} (y \sqrt{\mu_k}) \,\, , \,\, k \in \mathbb{N} \, ,
\end{equation*}
where $\eta_k$ is the $k^{\text{th}}$ positive root of $J_{-s}$ and the eigenvalues are given by
\begin{equation}
 \label{eq:evs-other}
 \mu_k = \eta_k^2 \mathcal{Y}^{-2} 
\end{equation}
\end{lemma}
\begin{proof}
We propose the ansatz $\psi(y)=z^s w(z)$, with $z = \sqrt{\mu} y$. The ODE becomes
\begin{equation*}
\label{eq:bessels-eq}
 z^s w''(z) + z w'(z) + (z^2 - s^2)w(z) = 0 \, .
\end{equation*}
This is a Bessel equation \cite[Section 9.1]{abramowitz1988handbook}, with general solution \cite[Section 10.2(ii)]{NIST:DLMF}
\begin{equation*}
 w(z) = A J_s (z) + B Y_s(z) \, .
\end{equation*}
$A,~B$ are constants; $J_s$ and $Y_s$ are the Bessel functions of the first and second kind.

We must now satisfy the boundary conditions. Substituting back that $\sqrt{\mu}y=z$, the Neumann boundary condition at $y=0$ becomes
\begin{align*}
  -\lim_{y \downarrow 0} y^\alpha \psi'(y) &= -\lim_{y \downarrow 0} y^\alpha \frac{d}{dy} \bigg[ (\sqrt{\mu} y)^s \big( A J_s(\sqrt{\mu}y) + B Y_s (\sqrt{\mu}y) \big)  \bigg] \\
   &= \frac{2^{1-s} \mu^s \big( A\pi + B\cos(\pi s) \Gamma(1-s) \Gamma(s) \big) }{\pi \Gamma(s)} = 0 \, ,
\end{align*}
i.e.,
$
  A\pi + B\cos( \pi s) \Gamma(1-s) \Gamma(s) = 0 \, .
$
The Dirichlet boundary condition implies that
\[
  AJ_s(\sqrt{\mu}\mathcal{Y}) + B Y_s(\sqrt{\mu}\mathcal{Y})=0 \, .
\]
Expressing $B$ in terms of $A$, it can be shown that
\begin{equation*}
 \psi(y ) = \frac{2^s A }{\Gamma(1-s) \cos( \pi s)} {}_0F_1 (0,1-s ; -\mu y^2 / 4) \, ,
\end{equation*}
where ${}_0F_1$ is the hypergeometric function. This exotic function can be expressed in terms of the
Bessel function of the first kind by the relation \cite[Equation (9.1.69)]{abramowitz1988handbook}
\begin{equation*}
 J_{-s} (\sqrt{\mu} y) = \frac{ \left( \frac{\sqrt{\mu} y}{2} \right)^{-s}}{\Gamma(1-s)} {}_0F_1 (0,1-s; -\mu y^2/4) \, .
\end{equation*}
So we may write
\begin{equation*}
 \psi(y) = \frac{A}{\cos(\pi s)} \left( \sqrt{\mu} y\right)^s J_{-s}(\sqrt{\mu} y)\, .
\end{equation*}
We can now choose the eigenvalue $\mu$ so that
$
  J_{-s}(\sqrt{\mu}\mathcal{Y}) = 0
$,
i.e., $\sqrt{\mu_k}\mathcal{Y} = \eta_k$. Finally, we choose the coefficient $A_k$ so that $\| \psi_k \|_{L^2(y^\alpha , (0,\mathcal{Y}))}^2=1$.
\end{proof}

According to \eqref{eq:sol-rep}, the eigenfunctions only need to be evaluated at $y=0$.

\begin{corollary}[value at $y=0$]
If $\{(\mu_k,\psi_k)\}_{k=1}^\infty$ solve \eqref{eq:eigen-bvp}, then
\begin{equation}
 \label{eq:efs-one-half-zero}
 \psi_k(0) = \sqrt{ \frac{2}{ \mathcal{Y} } } \,  \qquad \text{ when } \qquad s = \frac12 \, ,
\end{equation}
and
\begin{equation}
 \label{eq:efs-other-zero}
 \psi_k(0) = \frac{ 2^{s+1/2} }{\mu_k^{s/2} \mathcal{Y} J_{1-s}(\eta_k) \Gamma(1-s) } \, \qquad \text{ when } \qquad s \neq \frac12 \, .
\end{equation}
\end{corollary}
\begin{proof}
The case $s = \tfrac12$ is immediate.

If $s \neq \tfrac12$ the asymptotic for Bessel functions \cite[Equation (9.1.7)]{abramowitz1988handbook} is
\begin{equation*}
 J_{-s}(z) \sim \frac{1}{\Gamma(1-s)} \left( \frac{2}{z} \right)^s \, .
\end{equation*}
We use this to take the limit $y \downarrow 0$.
\end{proof}

Let us then summarize our approach. Given $f \in \mathbb{H}^{-s}(\Omega)$, we choose a finite element space $\mathbb{V}_h$.
We solve, for $k =1, \ldots, \mathcal{K}$, problem \eqref{eq:FindTheUks} where the constant $d_s$ is defined in \eqref{eq:define-ds}, $v_k(0) = \psi_k(0)$, and the pair $(\mu_k,\psi_k(0))$
is given by either \eqref{eq:evs-one-half} and \eqref{eq:efs-one-half-zero} or by \eqref{eq:evs-other} and \eqref{eq:efs-other-zero}. Finally, the approximate solution is
\[
  u_{h,\mathcal{Y}}^{\mathcal{K}} = \sum_{k=1}^{\mathcal{K}} \psi_k(0)U_k \in \mathbb{V}_h \, .
\]
The main computational appeal of the diagonalization developed in \cite{banjai2019tensor} remains: problems \eqref{eq:FindTheUks}
can be solved in parallel. In addition, the values of $\mu_k$ and $\psi_k(0)$ are now given explicitly. No unstable eigenvalue problem needs to be solved.

\section{Error Analysis}
\label{sec:quadrature-error}

Let us derive an error estimate for the semi-analytic diagonalization scheme. In order to do this, in what follows we shall assume
that $f \in L^2(\Omega)$.

\subsection{The discrete Laplacian}
\label{sub:Deltah}

We introduce the discrete Laplacian 
\[
  \Delta_h : \mathbb{V}_h \to \mathbb{V}_h
  \qquad
  \int_\Omega \Delta_h W V dx = - \int_\Omega \nabla W \cdot \nabla V dx, \quad \forall \, V,W \, \in \mathbb{V}_h \, .
\]
It is an invertible, symmetric, negative definite operator. 
Let $M = \dim \mathbb{V}_h$. We have:
\begin{enumerate}[1.]
  \item There are $\{ (\lambda_{h,m}, \Phi_{h,m}) \}_{m=1}^M \subset \mathbb{R}_+ \times \mathbb{V}_h$ such that,
  \[
    \int_\Omega (-\Delta_h) \Phi_{h,m} V dx = \int_\Omega \nabla \Phi_{h,m} \cdot \nabla V dx = \lambda_{h,m} \int_\Omega \Phi_{h,m} V dx, \quad \forall V \in \mathbb{V}_h \, .
  \]
  
  \item The eigenvectors are a basis of $\mathbb{V}_h$, and can be chosen to be orthonormal in $L^2(\Omega)$.
  
  \item The eigenvalues satisfy the bounds
  \begin{equation}\label{eq:discrete-eigenvals-bounds}
    C_P^{-1} \leq \lambda_{h,1} < \lambda_{h,2} \leq \cdots \leq \lambda_{h,M} \lesssim h^{-2}\, ,
  \end{equation}
  where $C_P$ is the best constant in the Poincar\'e inequality in $\Omega$. The lower bound follows
  from an application of the Rayleigh quotient. For the upper bound see \cite[Estimate (3.28)]{thomee2007galerkin}
  with $\beta=2$ and $\chi = \Phi_{h,M}$, where $\Phi_{h,M}$. 

  \item The fractional powers $(-\Delta_h)^{r}$, for $r \in \mathbb{R}$, are given by standard spectral theory.
\end{enumerate}

From these properties we have that, first, if $P_h$ is the $L^2(\Omega)$--projection onto $\mathbb{V}_h$,
\begin{equation}
\label{eq:LaphL2normestimate}
 \| (-\Delta_h) P_h\|_{L^2(\Omega) \to L^2(\Omega)}  \lesssim \frac{1}{h^2} \, .
\end{equation}
Next, since from the onset we have assumed that our domain is a convex polytope, we have that if $w \in H^1_0(\Omega)$ is such that $-\Delta w = g \in L^2(\Omega)$, then $w \in H^2(\Omega) \cap H^1_0(\Omega)$. Therefore, an application of Aubin-Nitsche duality \cite[Lemma 2.31]{ern2004theory} implies that,
\begin{equation*}
 \| w - w_{h} \|_{L^2(\Omega)} \lesssim h^2 |w|_{H^2(\Omega)} \lesssim h^2 \| g \|_{L^2(\Omega)} \, ,
\end{equation*}
where $w_{h} = (-\Delta_h)^{-1} P_h g$ is the Galerkin approximation to $w$. In operator terms:
\begin{equation}
\label{eq:OperatorError}
  \|  (-\Delta)^{-1} - (-\Delta_h)^{-1}  P_h \|_{L^2(\Omega) \to L^2(\Omega) } \lesssim h^2.
\end{equation}

\subsection{Error decomposition}

The solution to \eqref{eq:FindTheUks} can be written as
\begin{equation*}
 U_{k} = d_s \psi_k(0) (\mu_k \mathrm{I} - \Delta_h)^{-1} P_h f \, ,
\end{equation*}
where $\mathrm{I}$ is the identity operator and $P_h$ is the $L^2(\Omega)$--projection onto $\mathbb{V}_h$. Then,
\begin{equation}
 \label{eq:formal-solution-sum}
 u_{h,\mathcal{Y}}^{\mathcal{K}} = d_s \sum_{k=1}^{\mathcal{K}} |\psi_k(0)|^2 (\mu_k \mathrm{I} -\Delta_h)^{-1} P_h f \, .
\end{equation}
With this at hand we begin the error analysis by writing
\begin{equation}
 \label{eq:error-estimate}
 \| u - u_{h,\mathcal{Y}}^{\mathcal{K}} \|_{L^2(\Omega)} \leq \| u-u_h \|_{L^2(\Omega)} + \| u_h-u_{h,\mathcal{Y}}^\infty\|_{L^2(\Omega)}
   + \| u_{h,\mathcal{Y}}^\infty - u_{h,\mathcal{Y}}^{\mathcal{K}} \|_{L^2(\Omega)} \, .
\end{equation}
Here $u_h = (-\Delta_h)^{-s} P_h f$ and, as the notation suggests,
\[
  u_{h,\mathcal{Y}}^{\infty} = d_s \sum_{k=1}^\infty |\psi_k(0)|^2 (\mu_k \mathrm{I} -\Delta_h)^{-1} P_h f \, .
\]
For reasons that shall become clear during the course of our discussion, we shall call the first term on the right hand side of \eqref{eq:error-estimate} the
discretization error, the second one the quadrature error, whereas the last one we name the truncation error. Each of the three contributions of the error are
studied below.

\subsubsection{Discretization error}
We observe that \eqref{eq:LaphL2normestimate} and \eqref{eq:OperatorError} are the conditions of \cite[Theorem 1]{matsuki1993note}, and so we conclude that
\begin{equation}
 \label{eq:discretization-error}
 \| u - u_h \|_{L^2(\Omega)} \lesssim h^{2s} \| f \|_{L^2(\Omega)} \, .
\end{equation}

\subsubsection{Quadrature error: The Balakrishnan formula}

We recall the Balakrishnan formula for powers of positive operators; see \cite[Chapter 4]{lunardi2009interpolation} for the general
theory, and \cite{bonito2015numerical} and \cite{bonito2018numerical} for its application to the fractional Laplacian. As stated above, the operator $-\Delta_h$
is positive definite and symmetric. Therefore,
\begin{equation}
 \label{eq:balakrishnan}
 (-\Delta_h)^{-s} = \frac{2 \sin(\pi s)}{\pi} \int_0^\infty t^\alpha (t^2 \mathrm{I} - \Delta_h)^{-1} \, dt \, .
\end{equation}
We interpret \eqref{eq:formal-solution-sum} as a quadrature formula for the operator-valued weighted integral
\begin{equation*}
 \int_0^\infty t^\alpha F_h(t) \, dt \, , \qquad F_h(t) = \frac{2\sin(\pi s)}{\pi} (t^2\mathrm{I} - \Delta_h)^{-1} \, .
\end{equation*}

The nodes $\{t_k^{(s)}\}_{k=1}^\mathcal{K}$ and weights $\{\omega_k^{(s)}\}_{k=1}^\mathcal{K}$ of the quadrature scheme are dependent on the value of $s$. The specific values for the two cases, $s=\tfrac12$ and $s \neq \tfrac12$, will be given in \Cref{subsec:error-one-half} and \Cref{subsec:error-not-one-half}, respectively.

It is now clear why we name the terms on the right hand side as such. The difference $u_h - u_{h,\mathcal{Y}}^\infty$ encodes the error in
approximating the integral in \eqref{eq:balakrishnan} with a quadrature formula with an infinite number of nodes and weights. Finally, $u_{h,\mathcal{Y}}^\infty - u_{h,\mathcal{Y}}^{\mathcal{K}}$ measures the effect of truncating the quadrature rule
and only accounting for $\mathcal{K}$ nodes.

Let us introduce some notation. Given $\lambda >0$, we define the function
\[
  F_\lambda(t) = \frac{2 \sin(\pi s)}\pi\frac1{t^2+\lambda} \, .
\]
The integral and quadrature formulas of interest are then denoted by
\begin{equation}
\label{eq:Quad-Fla}
 I_s(F_\lambda) = \int_{0}^{\infty} t^{1-2s} F_\lambda(t)\, dt \, ,
  \qquad \qquad
 Q_s^{\mathcal{K}}(F_\lambda) = \sum_{k=1}^{\mathcal{K}} (t_k^{(s)})^{1-2s} \omega_k^{(s)} F_\lambda(t_k^{(s)}) \, ,
\end{equation}
respectively. The symbol $Q_s^\infty(F_\lambda)$ has then the expected meaning.

The main error estimate of our work is the content of the next result.

\begin{theorem}[error estimate]
\label{thm:convergence-rate-generic}
Let $\Omega \subset \mathbb{R}^d$ be a bounded, convex polytope and $f \in L^2(\Omega)$. Let, for $s \in (0,1)$,
$u \in \mathbb{H}^s(\Omega)$ be the solution to \eqref{eq:pde-problem} and $u_{h,\mathcal{Y}}^{\mathcal K} \in \mathbb{V}_h$ be defined in \eqref{eq:formal-solution-sum}. Denote $e = \| u - u_{h,\mathcal{Y}}^{\mathcal{K}}\|_{L^2(\Omega)}$. Then, we have
\[
e \lesssim  \left[
     h^{2s} 
     + \sup_{\lambda \geq \frac1{C_P}} \left| I_s(F_\lambda) - Q_s^{\infty}(F_\lambda) \right| 
+
     \sup_{\lambda \geq \frac1{C_P}}  \left| \sum_{k> \mathcal{K}} (t_k^{(s)})^{1-2s} \omega_k^{(s)} F_\lambda(t_k^{(s)})  \right|
 \right] \| f \|_{L^2(\Omega)} \, ,
\]
where the implicit constant is independent of $\mathcal Y$, $\mathcal K$, $h$, and $f$.
\end{theorem}
\begin{proof}
Owing to \eqref{eq:discretization-error} it suffices to study the quadrature and truncation errors.

Let $M = \dim \mathbb{V}_h$ and $\{ \Phi_{h,m} \}_{m=1}^M \subset \mathbb{V}_h$ be as in \Cref{sub:Deltah}. Then,
\[
  P_h f = \sum_{m=1}^M f_m\Phi_{h,m} \, , \qquad f_m = \int_\Omega f \Phi_{h,m} \, dx \, ,
\]
and
\[
  u_h = (-\Delta_h)^s P_h f = \int_0^\infty t^\alpha F_h(t) P_h f \, dt = \sum_{m=1}^M f_m I_s(F_{\lambda_{h,m}}) \Phi_{h,m} \, .
\]
Similarly, \eqref{eq:formal-solution-sum} can be rewritten as
\[
  u_{h,\mathcal{Y}}^{\mathcal K} = \sum_{m=1}^M f_m Q_s^{\mathcal K}(F_{\lambda_{h,m}}) \Phi_{h,m} \, .
\]

Since $\{\Phi_{h,m} \}_{m=1}^M$ is orthonormal in $L^2(\Omega)$,
\begin{align*}
  \| u_h - u_{h,\mathcal{Y}}^{\mathcal K} \|_{L^2(\Omega)}^2 &= \sum_{m=1}^M |f_m|^2 \left|  I_s(F_{\lambda_{h,m}}) -Q_s^{\mathcal K}(F_{\lambda_{h,m}}) \right|^2 \\
  &\leq \sup_{\lambda > \frac1{C_P}} \left|  I_s(F_{\lambda}) -Q_s^{\mathcal K}(F_{\lambda}) \right|^2 \| f \|_{L^2(\Omega)}^2 \, ,
\end{align*}
where we used \eqref{eq:discrete-eigenvals-bounds}. In summary we have obtained
\[
  \| u_h - u_{h,\mathcal{Y}}^{\mathcal K} \|_{L^2(\Omega)} \leq \sup_{\lambda > \frac1{C_P}} \left|  I_s(F_{\lambda}) -Q_s^{\mathcal K}(F_{\lambda}) \right| \| f \|_{L^2(\Omega)} \, .
\]
The result now immediately follows from
\[
  \left|  I_s(F_{\lambda}) -Q_s^{\mathcal K}(F_{\lambda}) \right| \leq \left|  I_s(F_{\lambda}) -Q_s^{\infty}(F_{\lambda}) \right| + \left| \sum_{k> \mathcal{K}}
    (t^{(s)})^{1-2s} \omega_k^{(s)} F_\lambda(t_k^{(s)})  \right| \, .
\]
\end{proof}

The usefulness of \Cref{thm:convergence-rate-generic} is clear. To obtain an error estimate it suffices to, depending on the value of $s$, estimate the two suprema. We now embark in this task.

\subsection{Error Analysis for $s=\tfrac12$}
\label{subsec:error-one-half}

In the case $s=\tfrac12$, we have that $d_s=1$. Substitute this, \eqref{eq:evs-one-half}, and \eqref{eq:efs-one-half-zero} into \eqref{eq:formal-solution-sum} to conclude that
in this case the quadrature formula $Q_{\tfrac12}^{\mathcal K}$, defined in \eqref{eq:Quad-Fla}, has the following nodes and weights:
\begin{equation}
 \label{eq:quad-n-w-one-half}
 t_k^{(\tfrac12)} = \frac{(k-\tfrac12)\pi}{\mathcal{Y}} \,\, , \qquad\qquad \omega_k^{(\tfrac12)} = \frac{\pi}{\mathcal{Y}} \, .
\end{equation}
In addition
\begin{equation}
 \label{eq:exact-int-one-half}
 I_{\tfrac12} (F_\lambda) = \int_0^\infty F_\lambda(t) \, dt =  \frac2\pi \int_0^\infty \frac{1}{t^2 + \lambda}\, dt =  \frac1{\sqrt{\lambda}}\, .
\end{equation}

We now give two lemmas that will aid in our error estimate.
\begin{lemma}[quadrature error]
\label{lemma:s12-quadrature}
Let $s=\tfrac12$ and $\lambda > 0$. If the weights and nodes of \eqref{eq:Quad-Fla} are specified by \eqref{eq:quad-n-w-one-half}, then we have
 \begin{equation}
  \left| I_{\tfrac12}(F_\lambda) - Q_{\tfrac12}^\infty (F_\lambda) \right| \lesssim \frac{\exp (-\sqrt{\lambda} \mathcal{Y})}{\sqrt{\lambda}} \, .
 \end{equation}
\end{lemma}
\begin{proof}
Set $a=\sqrt{\lambda} \mathcal{Y}/\pi$.
Using 
\cite[Equation (30:6:5)]{oldham2009atlas}, we write
 \begin{align*}
 Q_{\tfrac12}^\infty (F_\lambda) &= 
   \frac{\pi}{\mathcal{Y}} \sum_{k=1}^\infty \frac2\pi \frac{1}{ \frac{(k-\tfrac12)^2\pi^2}{\mathcal{Y}^2} + \lambda }
   = \frac2{\mathcal{Y}} \frac{\mathcal{Y}^2}{\pi^2} \sum_{k=1}^\infty \frac{1}{ (k-\tfrac12)^2 + \lambda \frac{\mathcal{Y}^2}{\pi^2}} \\
    &= \frac{8a}{\pi \sqrt{\lambda}} \sum_{k=1}^\infty \frac{1}{(2k-1)^2+(2a)^2}
    = \frac{8a}{\pi \sqrt{\lambda}} \frac{\pi \tanh(a\pi)}{8 a}
    = \frac1{\sqrt{\lambda}} \tanh \big( \sqrt{\lambda} \mathcal{Y} \big) \, .
 \end{align*}
 Combine this with \eqref{eq:exact-int-one-half} to obtain
 \begin{align*}
  \left| I_{\tfrac12}(F_\lambda) - Q_{\tfrac12}^\infty (F_\lambda) \right| &= \frac1{\sqrt{\lambda}} (1 - \tanh(\sqrt{\lambda} \mathcal{Y}))
    \lesssim \frac{\exp(-\sqrt{\lambda} \mathcal{Y})}{\sqrt{\lambda}} \, .
 \end{align*}
\end{proof}

\begin{lemma}[truncation error]
\label{lemma:s12-truncation}
 Let $s=\tfrac12$, $\lambda > 0$, and $\mathcal{K} \in \mathbb{N}$. Assume that the weights and nodes of $Q_{\tfrac12}^{\mathcal{K}}$ are given in \eqref{eq:quad-n-w-one-half}. If $\mathcal K$ is sufficiently large, then we have
 \begin{equation*}
   \left| Q_{\tfrac12}^\infty (F_\lambda) - Q_{\tfrac12}^{\mathcal{K}} (F_\lambda) \right| \lesssim \frac{\mathcal{Y}}{\mathcal{K}} \, .
 \end{equation*}
\end{lemma}
\begin{proof}
Recall that from \cite[Section 1.9, formula (10)]{bateman1953higher}
\begin{equation*}
  \psi' (\mathcal{K}) = \sum_{k \geq \mathcal{K}} \frac{1}{k^2} \, ,
\end{equation*}
where $\psi$ is the digamma function. Using the asymptotic bound for $\psi$ for large arguments \cite[Formula 6.4.12]{abramowitz1988handbook},
$\psi'$ behaves as
\begin{equation*}
  \psi'(x) \sim \frac{1}{x} + \frac{1}{2x^2} + \frac{1}{6x^3} - \frac{1}{30x^5} + \frac{1}{42 x^7} - \cdots \lesssim \frac1x \,\, .
\end{equation*}
Combining these two observations, we have that
\begin{align*}
  \left| Q_{\tfrac12}^\infty(F_\lambda) - Q_{\tfrac12}^{\mathcal K}(F_\lambda) \right| &= 
    \frac2{\mathcal Y} \sum_{k > \mathcal{K}}  \frac1{ \frac{(k-\tfrac12)^2 \pi^2}{{\mathcal Y}^2} + \lambda }
    = \frac{2 \mathcal{Y}}{\pi^2} \sum_{k> \mathcal{K}} \frac1{ (k-\tfrac12)^2 + \lambda \tfrac{{\mathcal Y}^2}{\pi^2} } \\
    & \leq \frac{2 \mathcal{Y}}{\pi^2} \sum_{k \geq \mathcal{K}} \frac1{ (k+\tfrac12)^2 } 
    \leq \frac{2 \mathcal{Y}}{\pi^2} \sum_{k \geq \mathcal{K}} \frac1{ k^2 } 
    = \frac{2 \mathcal{Y}}{\pi^2} \psi'( \mathcal K)
   \lesssim \frac{\mathcal{Y}}{\mathcal{K}} \, ,
\end{align*}
which is the bound we needed to obtain.
\end{proof}

Using these lemmas, we have a corollary to \Cref{thm:convergence-rate-generic}.

\begin{corollary}[error estimate for $s = \tfrac12$]
\label{cor:ErrEsts12}
Let $\Omega \subset \mathbb{R}^d$ be a bounded, convex polytope and $f \in L^2(\Omega)$. Let
$u \in \mathbb{H}^{1/2}(\Omega)$ solve \eqref{eq:pde-problem} for $s =\tfrac12$. Let $u_{h, \mathcal Y}^{\mathcal K} \in \mathbb{V}_h$ be given by \eqref{eq:formal-solution-sum}. Then, we have that
\begin{equation*}
 \| u - u_{h,\mathcal{Y}}^{\mathcal{K}} \|_{L^2(\Omega)} \lesssim \left( 
   h 
   + \sqrt{C_P} \exp \left( -\frac{\mathcal{Y}}{ \sqrt{C_P} } \right )
   + \frac{\mathcal{Y}}{\mathcal{K}}
 \right) \| f \|_{L^2(\Omega)} \, ,
\end{equation*}
where $C_P$ is given in \eqref{eq:discrete-eigenvals-bounds}. In particular, if $\mathcal Y \sim |\log h|$, and $\mathcal K \sim \frac{|\log h|}{h}$, we have
\[
  \| u - u_{h,\mathcal{Y}}^{\mathcal{K}} \|_{L^2(\Omega)} \lesssim h \| f \|_{L^2(\Omega)} \, .
\]
\end{corollary}
\begin{proof}
The proof of the error estimate essentially entails using \Cref{lemma:s12-quadrature} and \Cref{lemma:s12-truncation} to bound the estimate in \Cref{thm:convergence-rate-generic} and using \eqref{eq:discrete-eigenvals-bounds}.

The practical error estimate, that is the one in terms of $h$, is given by using the prescribed relations between all discretization parameters.
\end{proof}

\subsection{Error Analysis for $s \neq \tfrac12$}
\label{subsec:error-not-one-half}

We first use \eqref{eq:formal-solution-sum} to realize that this solution representation is indeed a quadrature formula for the Balakrishnan formula \eqref{eq:balakrishnan}.

\begin{lemma}[quadrature]
Let $s \in (0,1)\setminus \{\tfrac12\}$. Equation \eqref{eq:formal-solution-sum} can be written as
\[
  u_{h,\mathcal{Y}}^{\mathcal{K}} = \sum_{m=1}^M f_m Q_s^{\mathcal{K}} ( F_{\lambda_{h,m}} ) \Phi_{h,m} \, .
\]
The quadrature formula $Q_s^{\mathcal{K}}$ is defined in \eqref{eq:Quad-Fla}. Its nodes and weights are given by
\begin{equation}
  \label{eq:s-not-one-half-weights}
  t_k^{(s)} = \sqrt{\mu_k} = \frac{\eta_k}{\mathcal{Y}} \  , \qquad \omega_k^{(s)} = \frac{2}{\mathcal{Y} \eta_k J_{1-s}^2(\eta_k) } \, .
\end{equation}
\end{lemma}
\begin{proof}
We substitute \eqref{eq:define-ds} and \eqref{eq:efs-other-zero} into \eqref{eq:formal-solution-sum}. Using the famous Euler's reflection formula \cite[Formula 6.1.17]{abramowitz1988handbook}
\[
  \Gamma(s)\Gamma(1-s) = \frac\pi{\sin(\pi s)},
\]
we obtain
\begin{align*}
  u_{h,\mathcal{Y}}^{\mathcal{K}} &= d_s \sum_{k=1}^{\mathcal{K}} |\psi_k(0)|^2 (\mu_k \mathrm{I} - \Delta_h)^{-1}P_h f \\
   &= \frac{2^{1-2s} \Gamma (1-s)}{\Gamma(s)} \sum_{k=1}^{\mathcal{K}} \frac{2^{1+2s}}{ \mu_k^s \mathcal{Y}^2 \Gamma^2(1-s) J_{1-s}^2(\eta_k)} (\mu_k \mathrm{I} - \Delta_h)^{-1} P_h f \\
   &= \frac{2 \sin(\pi s)}{\pi} \sum_{k=1}^{\mathcal{K}} \frac{2}{\mu_k^s \mathcal{Y}^2 J_{1-s}^2(\eta_k)} (\mu_k \mathrm{I} - \Delta_h)^{-1} P_h f \, .
\end{align*}

Substitute \eqref{eq:evs-other} into this result to obtain that
\begin{align*}
  u_{h,\mathcal{Y}}^{\mathcal{K}} &=
    \sum_{k=1}^{\mathcal{K}} \left( \frac{\eta_k}{\mathcal{Y}} \right)^{1-2s} \frac{2}{\eta_k \mathcal{Y} J_{1-s}^2(\eta_k)} \left[ \frac{2 \sin(\pi s)}{\pi} \left( \left(\frac{\eta_k}{\mathcal Y} \right)^2 \mathrm{I} - \Delta_h \right)^{-1} \right] P_h f \\
  &= \sum_{k=1}^{\mathcal{K}} \left( \frac{\eta_k}{\mathcal{Y}} \right)^{1-2s} \frac{2}{\eta_k \mathcal{Y} J_{1-s}^2(\eta_k)} \left[ \frac{2 \sin(\pi s)}{\pi} \sum_{m=1}^M f_m \frac1{ \left(\frac{\eta_k}{\mathcal Y} \right)^2 + \lambda_{h,m} } \right] \Phi_{h,m} \\
  &= \sum_{m=1}^M f_m \sum_{k=1}^{\mathcal K} \left( t_k^{(s)} \right)^{1-2s} \omega_k^{(s)} F_{\lambda_{h,m}}(t_k^{(s)}) \Phi_{h,m}
  = \sum_{m=1}^M f_m Q_s^{\mathcal{K}} ( F_{\lambda_{h,m}} ) \Phi_{h,m} \, ,
\end{align*}
where, as claimed, the nodes and weights are given by \eqref{eq:s-not-one-half-weights}.
\end{proof}

Having identified the nodes and weights of our quadrature rule \eqref{eq:Quad-Fla}, we relate it to the one studied in \cite[Equations 1.1 and 1.2]{ogata2005numerical}. In this work, for $\nu>-1$, the following quadrature formula is studied
\[
  \int_{-\infty}^\infty |t|^{1+2\nu} G(t) \, dt \approx \tau \sum_{k \in \mathbb{Z}\setminus \{0\}} w_{\nu,k} | \tau \xi_{\nu,k} |^{1+2\nu} G( \tau \xi_{\nu,k}) \, ,
\]
where $\tau>0$ is a given parameter, the weights are given by
 \begin{equation*}
  w_{\nu,k} = \frac{2}{ \pi^2 \xi_{\nu, |k|} J_{\nu+1}^2 (\pi \xi_{\nu,|k|})}  \, , \quad k \in \mathbb{Z} \setminus \{0\} \, ,
 \end{equation*}
and, for $k \in \mathbb{N}$, $\xi_{\nu,k}$ are the zeros of $J_{\nu} (\pi x)$ ordered in such a way that
\[
  0 < \xi_{\nu, 1} < \xi_{\nu, 2} < \cdots.
\]
Finally, $\xi_{\nu, -k} = - \xi_{\nu, k}$, $k \in \mathbb{N}$.

Let $G$ be an even function. it is immediate that, with the same nodes and weights,
\begin{equation}
\label{eq:OgataQuad}
  \int_0^\infty t^{1+2\nu} G(t) dt \approx \tau \sum_{k =1}^\infty  w_{\nu,k} | \tau \xi_{\nu,k} |^{1+2\nu} G( \tau \xi_{\nu,k}) \, .
\end{equation}
We now relate this quadrature formula with $Q_s^\infty(F_\lambda)$.

\begin{lemma}[quadrature equivalence]
\label{lem:Quad-Equivalence}
 Let $s \in (0,1)\setminus \{ \tfrac12 \}$. Let the nodes and weights of the quadrature scheme $Q_s^\infty(F_\lambda)$ be given by \eqref{eq:s-not-one-half-weights}. Then, this quadrature coincides with the one given in \eqref{eq:OgataQuad} with $\nu = -s > -1$ and $\tau = \tfrac\pi{\mathcal Y}$.
\end{lemma}
\begin{proof}
First recall that, for any $\lambda >0$, the function $F_\lambda$ is even. Next, for $k \in \mathbb{N}$, we have that $J_{-s}(\eta_k)=0$. We thus conclude that $\pi \xi_{-s,k} = \eta_k$. In addition, by setting $\tau = \frac{\pi}{\mathcal{Y}}$, we have $\frac{\eta_k}{\mathcal{Y}} = \tau \xi_{-s,k}$. Substituting this into \eqref{eq:OgataQuad} we see that
 \begin{align*}
  \tau \sum_{k=1}^{\mathcal K} \frac{2}{\pi^2 \xi_{-s,k} J_{1-s}^2(\pi \xi_{-s,k})} & | \tau \xi_{-s,k} |^{1-2s} F_\lambda( \tau \xi_{-s,k}) \\
   &= \sum_{k =1}^{\mathcal K} \frac{\tau}{\pi} \frac{2}{\pi \xi_{-s,k} J_{1-s}^2(\pi \xi_{-s,k}) } | \tau \xi_{-s,k}|^{1-2s} F_\lambda( \tau \xi_{-s,k} ) \\
   &= \sum_{k =1}^{\mathcal K} \frac{2}{\mathcal{Y} \eta_k J_{1-s}^2(\eta_k)} \left(\frac{\eta_k}{\mathcal{Y}} \right)^{1-2s} F_\lambda \left(\frac{\eta_k}{\mathcal{Y}} \right) 
   = Q_s^\infty( F_\lambda ) \, .
 \end{align*}
\end{proof}

Reference \cite{ogata2005numerical} provides an analysis for the quadrature \eqref{eq:OgataQuad}. To describe it, following \cite[Definition 2.1]{ogata2005numerical} we introduce the following class.

\begin{definition}[class $\mathcal{B}_{s,\ell}$]
\label{def:ClassBsd}
Let $\ell>0$ and denote
\[
  D_\ell = \{ z \in \mathbb{C} : | \Im z| < \ell \} \, , \qquad \Gamma_\ell = \partial D_\ell \, .
\]
By $\mathcal{B}_{s,\ell}$ we denote the collection of functions $G: \bar{D}_\ell \to \mathbb{C}$ that satisfy:
\end{definition}
\begin{enumerate}[1.]
 \item $G$ is analytic in the strip $D_\ell$.
 
 \item For all $0<c<\ell$, the integral
 \begin{equation*}
  \mathcal{N}_{s,c}(G) = \int_{-\infty}^\infty \left[ |x+\imath c|^{1-2s} |G(x+\imath c)| + |x- \imath c|^{1-2s} |G(x-\imath c)| \right] \, dx
 \end{equation*}
 exists. In addition, we require that $\mathcal{N}_{s,\ell-0}(G) = \lim_{c \uparrow \ell} \mathcal{N}_{s,c}(G)$ exists and is finite.

  \item For all $0<c<\ell$, we require that
 \begin{equation}
  \label{eq:condition-integral-ii}
  \lim_{x \to \pm \infty} \int_{-c}^c |x+\imath y|^{1-2s} |G(x+\imath y)| \, dy = 0 \, .
 \end{equation}
\end{enumerate}

The usefulness of this class with regards to quadrature is the following result.

\begin{theorem}[quadrature error estimate]
\label{thm:ogata}
 Let $\ell > 0$, $s \in (0,1)$, and $G \in \mathcal{B}_{s, \ell}$. Then
 \begin{equation}
 \label{eq:ogata-bound}
  | I_{s} (G) - Q_{s}^\infty (G) | \lesssim \mathcal{N}_{s,\ell-0}(G) \exp \left( -2\ell\mathcal{Y} \right) \, ,
 \end{equation}
where the implicit constant depends only on $s$ and $\ell$.
\end{theorem}
\begin{proof}
  See \cite[Theorem 2.1]{ogata2005numerical}, and use the equivalence of \Cref{lem:Quad-Equivalence}.
\end{proof}

The functions we are integrating are uniformly in a class $\mathcal{B}_{s,\ell}$.

\begin{theorem}[$F_\lambda \in \mathcal{B}_{s,\ell}$]
\label{thm:F-in_B}
  Let $s \in (0,1)$, $\lambda_0 > 0$, and set $\ell = \tfrac12 \sqrt{\lambda_0}$. For every $\lambda \geq \lambda_0$ we have $F_\lambda \in \mathcal{B}_{s,\ell}$. Moreover,
  $
    \mathcal{N}_{s,\ell-0}(F_\lambda)
  $
  depends only on $s$ and $\ell$.
\end{theorem}
\begin{proof}
For this proof we set $a_s = \tfrac{2\sin(\pi s)}\pi$. We verify each one of the conditions of \Cref{def:ClassBsd} below.
\begin{enumerate}[1.]
  \item We first observe that the only poles of $F_\lambda$ are $\pm \imath \sqrt{\lambda}$. Therefore, by setting $\ell = \tfrac12 \sqrt{\lambda_0}$ we see that, for every $\lambda \geq \lambda_0$, $F_\lambda$ is analytic in $D_\ell$.

  \item Now let $c \in (0,\ell)$. We have
  \begin{align*}
    |F_{\lambda} (x+&\imath c)|^2 = \left| \frac{a_s}{x^2 - c^2 + \lambda + 2 \imath cx} \right|^2 = \frac{a_s^2}{|x^2-c^2+ \lambda +2\imath cx|^2} \\
    &= \frac{a_s^2}{(x^2-c^2+\lambda)^2+(2cx)^2}
     = \frac{a_s^2}{x^4 +2c^2 x^2 + 2 x^2 \lambda -2c^2 \lambda +c^4 + \lambda}.
  \end{align*}
  Then, for $x > c$,
  \begin{align*}
   |x+&\imath c|^{1-2s} |F_{\lambda} (x+\imath c)| = \frac{ a_s (x^2 + c^2 )^{\frac{1-2s}{2}} }{ \big( (x^2-c^2+\lambda)^2 + 4c^2 x^2 \big)^{\tfrac12} } \\
   &= \frac{a_s x^{1-2s} \left( 1 + \frac{c^2}{x^2}\right)^{\frac{1-2s}{2}} }
     {x^2 \left( 1 + \frac{2c^2}{x^2} + \frac{2\lambda}{x^2} - \frac{2c^2 \lambda}{x^4} + \frac{c^4}{x^4} + \frac{\lambda^2}{x^4} \right)^{\tfrac12}}
   \leq a_s x^{-1-2s} \left(1 + \frac{c^2}{x^2} \right)^{\frac{1-2s}{2}} ,
  \end{align*}
  where we also used that $c < \frac{\sqrt{\lambda_0}}{2} \leq \frac{\sqrt{\lambda}}{2}$. Then,
  \begin{multline*}
   \frac1{a_s} \int_{-\infty}^{\infty} |x+ \imath c|^{1-2s}|F_{\lambda}(x+\imath c)| \, dx \leq  \\
     \int_{-\infty}^{-c^2} |x|^{-1-2s} \left( 1 + \frac{c^2}{x^2}\right)^{\frac{1-2s}{2}} \, dx  \, + \,
     C + \int_{c^2}^\infty x^{-1-2s} \left( 1 + \frac{c^2}{x^2}\right)^{\frac{1-2s}{2}} \, dx \, .
  \end{multline*}
  To obtain that
$
   \int_{c^2}^\infty x^{-1-2s} \, dx < \infty \, ,
 $
  we must have that $-1-2s<-1$, or $s >0$. A similar analysis holds for $|x-\imath c|^{1-2s}|F_{\lambda}(x-\imath c)|$. Hence, $\mathcal{N}_{s,c}(F_\lambda)$ exists.
  
  Notice, finally, that we have $c \leq \ell = \tfrac{\sqrt\lambda}2 < \sqrt\lambda$. In other words, we are well separated from the poles of $F_\lambda$. Consequently,
  \[
    \mathcal{N}_{s, \ell - 0}(F_\lambda) = \lim_{c \uparrow \ell} \mathcal{N}_{s, c}(F_\lambda)
  \]
  exists, is finite, and only depends on $s$ and $\ell$.
  
  \item We now check the last condition given by \eqref{eq:condition-integral-ii}. Suppose that $x > 2c$. Then,
  \begin{align*}
   \int_{-c}^c |x+ &\imath y|^{1-2s} \left| \frac{1}{(x+\imath y)^2+\lambda }\right| \, dy = 
     2 \int_0^c \frac{ (x^2+y^2)^{\tfrac{1-2s}{2}} }{\sqrt{|x^2+2\imath xy-y^2+\lambda|} } \, dy \\
     &= 2 \int_0^c \sqrt{ \frac{ (x^2+y^2)^{1-2s}}{ (x^2+\lambda -y^2)^2 + 4x^2 y^2 } } \, dy
     \leq 2 \int_0^c \frac{ \sqrt{ (x^2 + y^2)^{1-2s}} }{ x^2 + \lambda - y^2} \, dy  \\
     &= \frac{2 x^{1-2s}}{x} \int_0^c \frac{ \left( 1 + \frac{y^2}{x^2} \right)^{\frac{1-2s}{2}} }{ 1 + \frac{\lambda}{x^2}-\frac{y^2}{x^2}} \, \frac{dy}{x}
     \leq 2x^{-2s} \int_0^{c/x} \frac{ (1+t^2)^{\frac{1-2s}{2}} }{1-t^2} \, dt  \, ,
  \end{align*}
  where we used the change of variables $t = \frac{y}{x}$. Since $0<2y < 2c < x$, we conclude that $t = \tfrac{y}{x} \in ( 0,\tfrac{1}{2})$, which implies $1- t^2 > \frac{3}{4}$. Therefore,
  \[
    \int_{-c}^c |x+\imath y|^{1-2s} \left| \frac{1}{(x+\imath y)^2+\lambda}\right| \, dy \leq \frac{8}{3} x^{-2s} \int_0^{c/x} (1+t^2)^{\tfrac{1-2s}{2}} \, dt\, .
  \]
  We now proceed depending on the sign of $1-2s$. If $1-2s>0$, then
  \[
    \int_{-c}^c |x+\imath y|^{1-2s} \left| \frac{1}{(x+\imath y)^2+\lambda}\right| \, dy \leq \frac{8}{3} x^{-2s} \left( \frac54 \right)^{\tfrac{1-2s}{2}} \frac{c}x \, ,
  \]
  whereas if $1-2s<0$,
  \[
    \int_{-c}^c |x+\imath y|^{1-2s} \left| \frac{1}{(x+\imath y)^2+\lambda}\right| \, dy \leq \frac{8}{3} x^{-2s} \frac{c}x \, .
  \]
  In any event, these estimates are sufficient for us to conclude that
  \[
    \lim_{x \to \pm \infty} \int_{-c}^c |x+\imath y|^{1-2s} |F_\lambda(x + \imath y )| \, dy  = 0 \, .
  \]
\end{enumerate}
All requirements have been verified and the proof is complete.
\end{proof}

An immediate consequence is an estimate for the quadrature error.

\begin{corollary}[quadrature error]
\label{col:QuadBesselError}
Let $s \in (0,1)\setminus\{\tfrac12\}$. Let the weights and nodes of $Q_s^\infty$ be given by \eqref{eq:s-not-one-half-weights}. Then
\[
  \sup_{\lambda \geq \tfrac1{C_P}} |I_s(F_\lambda) - Q_s^\infty(F_\lambda)| \lesssim \exp \left( - \frac{\mathcal{Y}}{\sqrt{C_P}} \right) \, ,
\]
where $C_P$ is the best constant in the Poincar\'e inequality; see \eqref{eq:discrete-eigenvals-bounds}.
\end{corollary}
\begin{proof}
It suffices to set $\lambda_0 =\tfrac1{C_P}$ in \Cref{thm:F-in_B}.
\end{proof}

It remains to bound the truncation error.

\begin{lemma}[truncation error]
 \label{lemma:bessel-quadrature}
 Let $s \in (0,1) \setminus \{ \tfrac12 \}$ and $\mathcal{K} \in \mathbb{N}$. Let the nodes and weights of $Q_s^{\mathcal K}$ be given by \eqref{eq:s-not-one-half-weights}. If $\mathcal{K}$ is sufficiently large, then 
 \begin{equation*}
   \sup_{ \lambda > \tfrac1{C_P} } |Q_{s}^\infty (F_{\lambda}) - Q_{s}^{\mathcal{K}} (F_{\lambda}) | \lesssim \left( \frac{{\mathcal{Y}}}{\mathcal{K}} \right)^{2s} \, .
 \end{equation*}
\end{lemma}
\begin{proof}
The asymptotic \cite[Formula 9.2.1]{abramowitz1988handbook}
\begin{equation*}
 J_{-s}(z) \sim \sqrt{ \frac{2}{\pi z}} \, ,
\end{equation*}
which is valid for large arguments, allows us to estimate
\begin{align*}
 | Q_{s}^\infty (F_{\lambda}) - Q_{s}^{\mathcal{K}} (F_{\lambda}) | &\leq \frac{\pi}{\mathcal{Y}} \sum_{k > \mathcal{K}} \frac{2}{\pi \eta_k J_{1-s}(\eta_k)^2 }
   \left( \frac{\eta_k}{\mathcal{Y}} \right)^{1-2s} F_{\lambda} \left(\frac{\eta_k}{\mathcal{Y}} \right)  \\
   &= \frac{4\sin(\pi s)}{\mathcal{Y}}  \sum_{k > \mathcal{K}} \frac{1}{\pi \eta_k J_{1-s}(\eta_k)^2} \left( \frac{\eta_k}{\mathcal{Y}} \right)^{1-2s}
     \frac{1}{ \frac{\eta_k^2}{\mathcal{Y}^2} + \lambda } \\
   &\lesssim  \frac1{\mathcal{Y}} \sum_{k > \mathcal{K}} \left(\frac{\eta_k}{\mathcal{Y}}\right)^{1-2s} \frac{1}{\frac{\eta_k^2}{\mathcal{Y}^2}+\lambda }
   = \mathcal{Y}^{2s} \sum_{k > \mathcal{K}} \frac{\eta_k^{1-2s}}{\eta_k^2+\mathcal{Y}^2 \lambda} \, .
\end{align*}

McMahon’s asymptotic for large zeros \cite[Section 10.21 (vi)]{NIST:DLMF},
$
 \eta_k \sim k\pi \, ,
$
gives
\[
  \sum_{k>\mathcal{K}} \frac{\eta_k^{1-2s}}{\eta_k^2+\mathcal{Y}^2 \lambda} \lesssim \sum_{k>\mathcal{K}} \frac1{k^{1+2s}} = \zeta( 1+2s, \mathcal{K} +1 ) \, ,
\]
where by $\zeta$ we denote the Hurwitz Zeta function \cite[Section 25.11]{NIST:DLMF}. Using that, for large $\mathcal{K}$, we have the asymptotic \cite[Equation 25.11.43]{NIST:DLMF}
\[
  \zeta(1+2s, \mathcal{K}+1 ) \sim (\mathcal{K}+1 )^{-2s} < \mathcal{K}^{-2s} \, ,
\]
the claimed result follows.
\end{proof}

With \Cref{col:QuadBesselError} and \Cref{lemma:bessel-quadrature} at hand, we make a final corollary to \Cref{thm:convergence-rate-generic} and provide an error estimate for $s \neq \tfrac12$.

\begin{corollary}[error estimate for $s \neq \tfrac12$]
Let $\Omega \subset \mathbb{R}^d$ be a bounded, convex polytope and $f\in L^2(\Omega)$. Let $u \in \mathbb{H}^s(\Omega)$
be the solution to \eqref{eq:pde-problem} for $s \in (0,1) \setminus \{\tfrac12\}$.
Let $u_{h,\mathcal{Y}}^{\mathcal{K}} \in \mathbb{V}_h$ be given by \eqref{eq:formal-solution-sum}. Then, we have that
\begin{equation*}
 \| u - u_{h,\mathcal{Y}}^{\mathcal{K}} \|_{L^2(\Omega)} \lesssim \left( h^{2s} + \exp
   \left( -\frac{\mathcal{Y}}{\sqrt{C_P}} \right) + \left( \frac{\mathcal{Y}}{\mathcal{K}} \right)^{2s} \right)
 \| f \|_{L^2(\Omega)} \, ,
\end{equation*}
where $C_P$ is given by \eqref{eq:discrete-eigenvals-bounds}.
In particular, if $\mathcal{Y} \sim 2s | \log h|$ and $\mathcal{K} \sim \frac{\mathcal{Y}}{h}$, we have
\begin{equation*}
 \| u - u_{h,\mathcal{Y}}^{\mathcal{K}} \|_{L^2(\Omega)} \lesssim h^{2s} \| f \|_{L^2(\Omega)} \, .
\end{equation*}
\end{corollary}
\begin{proof}
One merely needs to apply \Cref{col:QuadBesselError} and \Cref{lemma:bessel-quadrature} to each one of the terms in the conclusion of \Cref{thm:convergence-rate-generic}.

The practical error estimate, that is the one given solely in terms of $h$, is given by using the prescribed relations between all discretization parameters.
\end{proof}

\section{Numerical Results}
\label{sec:results}

Previous work demonstrated the convergence of the tensor FEM of the Caffarelli-Silvestre
extension problem without diagonalization \cite{nochetto2015pde} and with diagonalization
\cite{banjai2019tensor}. To validate the convergence rates for our exact diagonalization technique, similar numerical experiments are conducted.
The method is implemented using the \texttt{deal.ii} finite element library \cite{dealII95,dealii2019design}, which
uses of a wide array of scientific libraries including the Message Passing Interface (MPI) for
parallel computation. With MPI interface in \texttt{deal.ii}, the parallel performance of the
method is measured. Both strong scaling and weak scaling performances are measured. All meshes in $\Omega$
are uniform, and we use conforming $\mathbb{Q}_1$ elements. Following \Cref{sec:quadrature-error}, the truncation height
$\mathcal{Y}$ and number of eigenpairs $\mathcal K$ are chosen as
\begin{equation}
 \label{eq:recipe}
 \mathcal{Y} = 2s | \log (h) | \,\, , \qquad\qquad \mathcal{K} = \frac{2s |\log (h)|}{h} \, .
\end{equation}

\subsection{Convergence Studies}
\label{subsec:convergence}

The numerical solution is computed
to approximate the solution to two smooth solutions. For the first problem, we consider an $L$-shaped domain,
$\Omega_L \subset \mathbb{R}^2$, determined by the vertices
\[
  \mathbf{c} \in \{ (0,0) \, , (1,0) \, , (-1,1) \, , (-1,-1)\, , (0,-1) \, \} \, .
\]
The exact solution is
\begin{equation*}
  u(x_1, x_2) = \sin (\pi x_1) \sin (\pi x_2) + \sin( 3\pi x_1) \sin( 2 \pi x_2)
    + \sin(5\pi x_1) \sin( 4 \pi x_2 ) \, ,
\end{equation*}
and the right hand side is chosen accordingly. Notice that, since the domain is not convex, our theoretical developments do not apply as written. However, the only place where convexity is invoked is in the derivation of \eqref{eq:OperatorError}. One can easily develop a more general error analysis that takes into account a smaller regularity shift.

A $\mathbb{Q}_1$-FEM on uniformly refined meshes in $\Omega_L$ is used. For the parameters we follow \eqref{eq:recipe}. The results can be seen in \Cref{tab:convergence-l-shaped}. The
observed convergence rates, for $s=0.25$ and $s=0.75$, match the theoretical rate of $2s$.
The error is also shown in \Cref{fig:lshaped-convergence} for both $s=0.25$ and $s=0.75$ along with the corresponding theoretical convergence lines.

\begin{table}[tbhp]
\scriptsize
\captionsetup{position=top} 
\caption{Number of elements for the $\mathbb{Q}_1$ FEM on uniform meshes in $\Omega_L$, height $\mathcal{Y}$ of the truncation, $\mathcal{K}$ number of eigenpairs used, the resulting error in the $L^2(\Omega_L)$ norm, and the convergence rate for $s=0.25$ and $s=0.75$.}
\label{tab:convergence-l-shaped}
\begin{center}
  \begin{tabular}{|r|c|r|c|c|c|r|c|c|} \hline
    \multirow{2}{4em}{\# $\mathcal{T}_{\Omega_L}$} &  \multicolumn{4}{c}{$s=0.25$} &  \multicolumn{4}{|c|}{$s=0.75$} \\
     & $\mathcal{Y}$ &  $\mathcal{K}$ & Error & Rate & $\mathcal{Y}$ & $\mathcal{K}$ & Error & Rate \\ \hline
       65  &  0.69315  &     2  &  1.1847e+00  &  --    &  2.07944  &     8  &  8.62909e-01  &  --    \\
      225  &  1.03972  &     8  &  9.22409e-01  &  0.36  &  3.11916  &    24  &  3.54492e-01  &  1.28  \\
      833  &  1.38629  &    22  &  6.59965e-01  &  0.48  &  4.15888  &    66  &  1.15435e-01  &  1.62  \\
     3201  &  1.73287  &    55  &  4.68283e-01  &  0.50  &  5.19860  &   166  &  3.69013e-02  &  1.65  \\
    12545  &  2.07944  &   133  &  3.29807e-01  &  0.51  &  6.23832  &   399  &  1.19066e-02  &  1.63  \\
    49665  &  2.42602  &   310  &  2.33219e-01  &  0.50  &  7.27805  &   931  &  3.91190e-03  &  1.61  \\
   197633  &  2.77259  &   709  &  1.64819e-01  &  0.50  &  8.31777  &  2129  &  1.30633e-03  &  1.58  \\
   788481  &  3.11916  &  1597  &  1.16468e-01  &  0.50  &  9.35749  &  4791  &  4.42652e-04  &  1.56  \\
   3149825 &  3.46574  &  3548  &  8.23623e-02  &  0.50  &  10.3972  & 10646  &  1.51727e-04  &  1.54  \\
    \hline
  \end{tabular}
 \end{center}
\end{table}

\begin{figure}[tbhp]
 \centering
   \begin{tikzpicture}
    \begin{axis}[
        title={$L^2(\Omega_L)$ convergence},
       xlabel={Degrees of Freedom},
       ylabel={$L^2(\Omega_L)$ error},
       xmode=log,
       ymode=log,
       xmin=0, xmax=4000000,
       ymin=0.0001, ymax=2,
       legend pos=south west,
       ymajorgrids=true,
       grid style=dashed,
     ]
     \addplot[
     color=black,
     mark=none,
     ]
     coordinates {
      (65,2)(225,1.4142)(833,1)(3201,0.70711)(12545,0.5)(49665,0.35355)(197633,0.25)(788481,0.17678)(3149825,0.1250)
     };
     \legend{$\text{error }\propto h^{2s}=h^{0.5}$}

    \addplot[
      color=blue,
      mark=square,
    ]
    coordinates {
    (65,1.18473)(225,0.922409)(833,0.659965)(3201,0.468283)(12545,0.329807)(49665,0.233219)(197633,0.164819)(788481,0.116468)(3149825,0.0823623)
    };
    \addlegendentry{$s=1/4$}

    \addplot[
      color=red,
      mark=none,
    ]
    coordinates {
    (65,1)(225,0.35255)(833,0.125)(3201,0.0442)(12545,0.01563)(49665,0.0055)(197633,0.0020)(788481,0.0007)(3149825,0.000248)
    };
    \addlegendentry{$\text{error }\propto h^{2s}=h^{1.5}$}

    \addplot[
      color=teal,
      mark=square,
    ]
    coordinates {
    (65,0.862909)(225,0.354492)(833,0.115435)(3201,0.0369013)(12545,0.0119066)(49665,0.00391190)(197633,0.00130633)(788481,0.000442652)(3149825,0.00015173)
    };
    \addlegendentry{$s=3/4$}
 \end{axis}
 \end{tikzpicture}
 \caption{Error in the $L^2(\Omega_L)$ norm versus the number of degrees of freedom using $\mathbb{Q}_1$ finite elements for $s=1/4$ and $s=3/4$ on uniformly refined meshes of $\Omega_L$.}
 \label{fig:lshaped-convergence}
\end{figure}
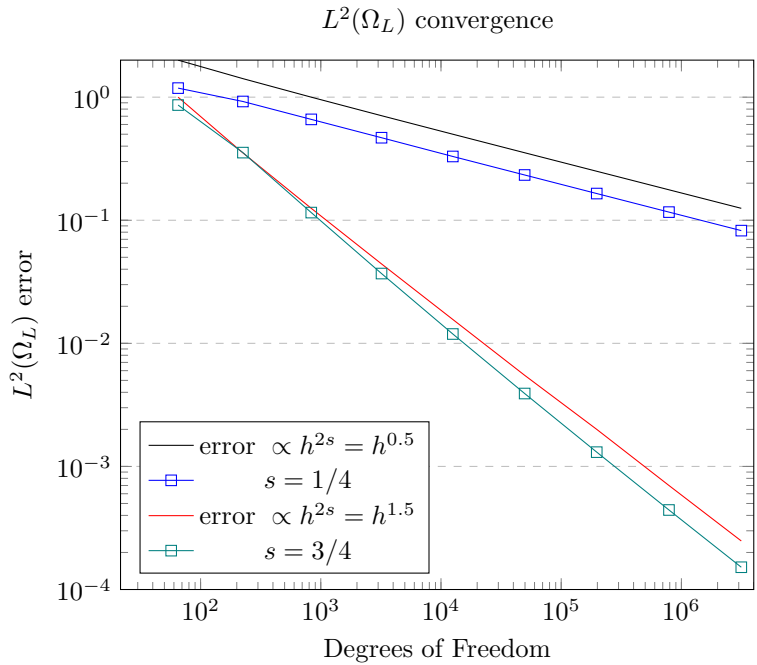

For the second problem, we consider a circular domain, i.e.,
\[
  \Omega_C = \{ x \in \mathbb{R}^2 : |x| < 1 \} \, .
\]
Using polar coordinates, it can be shown that
\begin{equation*}
 \phi_{m,n} (r , \theta) = J_m (j_{m,n} r) \big( A_{m,n} \cos (m\theta) + B_{m,n} \sin(m \theta) \big) \, ,
\end{equation*}
where $J_m$ is the Bessel function of the first kind with parameter $m$, $j_{m,n}$ is the $n$\textsuperscript{th} zero of $J_m$,
and $A_{m,n}$, $B_{m,n}$ are normalization constants.
It may also be shown that $\lambda_{m,n} = (j_{m,n})^2$.
We set $f=(\lambda_{1,1})^s \phi_{1,1}$. Then by Equation \eqref{eq:spectral-fractional-laplacian},
\begin{equation*}
 u(r, \theta) = \phi_{1,1} (r, \theta) \, .
\end{equation*}

As in the previous case, a $\mathbb{Q}_1$-FEM on uniformly refined meshes in $\Omega_C$ is used.
The formulas given in \eqref{eq:recipe} are used. The results are presented in
\Cref{tab:convergence-circular}. 
The observed convergence rates, for $s\in \{0.25, 0.75\}$, agree 
with the theoretical rate of $2s$.
\Cref{fig:circular-convergence} shows the error for both $s=0.25$ and $s=0.75$ next to the corresponding theoretical error convergence lines.

\begin{table}[tbhp]
\scriptsize
\captionsetup{position=top} 
\caption{Number of elements for the $\mathbb{Q}_1$ FEM on uniform meshes in $\Omega_C$, height $\mathcal{Y}$ of the truncation, $\mathcal{K}$ number of eigenpairs used, the resulting error in the $L^2(\Omega_C)$ norm, and the convergence rate for $s=0.25$ and $s=0.75$.}
\label{tab:convergence-circular}
\begin{center}
  \begin{tabular}{|r|c|r|c|c|c|r|c|c|} \hline
    \multirow{2}{4em}{\# $\mathcal{T}_{\Omega_C}$} &  \multicolumn{4}{c}{$s=0.25$} &  \multicolumn{4}{|c|}{$s=0.75$} \\
     & $\mathcal{Y}$ &  $\mathcal{K}$ & Error & Rate & $\mathcal{Y}$ & $\mathcal{K}$ & Error & Rate \\ \hline
     89  &  0.69315  &     2  &  3.95613e-01  &  -- &  2.07944  &     8  &  6.73748e-02  &  -- \\
     337  &  1.03972  &     8  &  2.49627e-01  &  0.66 &  3.11916  &    24  &  2.09143e-02  &  1.69 \\
     1313  &  1.38629  &    22  &  1.74086e-01  &  0.52  &  4.15888  &    66  &  6.38904e-03  &  1.71 \\
     5185  &  1.73287  &    55  &  1.23124e-01  &  0.50 &  5.19860  &   166  &  2.04193e-03  &  1.65 \\
     20609  &  2.07944  &   133  &  8.67400e-02  &  0.51    &  6.23832  &   399  &  6.70252e-04  &  1.61 \\
     82177  &  2.42602  &   310  &  6.13695e-02  &  0.50  &  7.27805  &   931  &  2.24480e-04  &  1.58 \\
     328193  &  2.77259  &   709  &  4.33828e-02  &  0.50 &  8.31777  &  2129  &  7.62235e-05  &  1.56 \\
     1311745  &  3.11916  &  1597  &  3.06600e-02  &  0.50 &  9.35749  &  4791  &  2.61759e-05  &  1.54 \\
    \hline
  \end{tabular}
 \end{center}
\end{table}

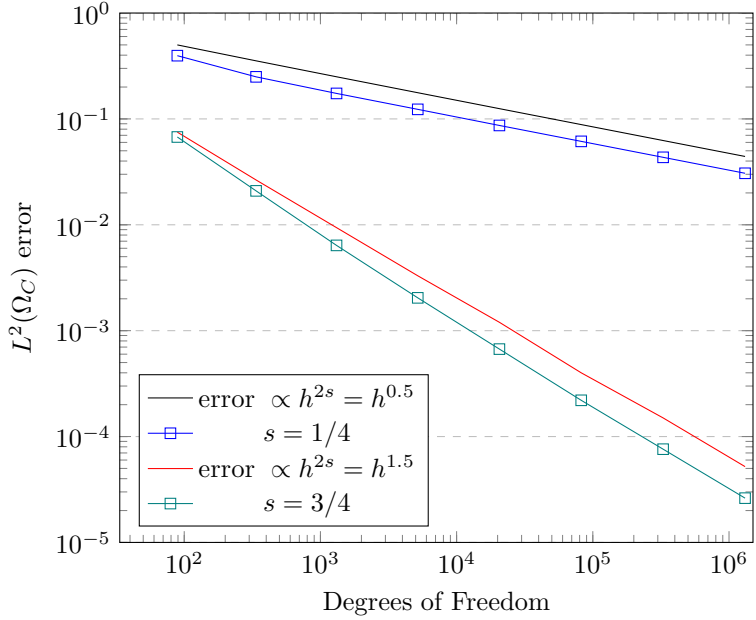
\begin{figure}[tbhp]
 \centering
   \begin{tikzpicture}
     \begin{axis}[
       xlabel={Degrees of Freedom},
       ylabel={$L^2(\Omega_C)$ error},
       xmode=log,
       ymode=log,
       xmin=0, xmax=1500000,
       ymin=0.00001, ymax=1,
       legend pos=south west,
       ymajorgrids=true,
       grid style=dashed,
     ]
     \addplot[
     color=black,
     mark=none,
     ]
     coordinates {
      (89,0.5)(337,0.3536)(1313,0.25)(5185,0.1768)(20609,0.125)(82177,0.0884)(328193,0.0625)(1311745,0.0442)
     };
     \legend{$\text{error }\propto h^{2s}=h^{0.5}$}

    \addplot[
      color=blue,
      mark=square,
    ]
    coordinates {
    (89,0.395613)(337,0.249627)(1313,0.174086)(5185,0.123124)(20609,0.08674)(82177,0.06137)(328193,0.04338)(1311745,0.03066)
    };
    \addlegendentry{$s=1/4$}

    \addplot[
      color=red,
      mark=none,
    ]
    coordinates {
    (89,0.075)(337,0.0265)(1313,0.0094)(5185,0.0033)(20609,0.0012)(82177,0.0004)(328193,0.00015)(1311745,0.000052)
    };
    \addlegendentry{$\text{error }\propto h^{2s}=h^{1.5}$}

    \addplot[
      color=teal,
      mark=square,
    ]
    coordinates {
    (89,0.06748)(337,0.02091)(1313,0.00639)(5185,0.002042)(20609,0.00067)(82177,0.00022)(328193,0.000076)(1311745,0.0000262)
    };
    \addlegendentry{$s=3/4$}

 \end{axis}
\end{tikzpicture}
\caption{Error in the $L^2(\Omega_C)$ norm versus the number of degrees of freedom using $\mathbb{Q}_1$ finite elements for $s=1/4$ and $s=3/4$ on uniformly refined meshes of $\Omega_C$.}
\label{fig:circular-convergence}
\end{figure}

\subsection{Parallel Performance}
\label{subsec:parallel-perf}

In \Cref{sec:diagonalization}, it is stated that the diagonalization technique exposes an inherent parallelization
opportunity for the numerical algorithm. We measure the parallel performance in two scaling tests. For both tests, we let
$\Omega_\square = (0,1)^2$. In this case, it is known that
\begin{equation*}
 \phi_{m,n}(x_1, x_2) = \sin(m \pi x_1) \sin(n \pi x_2) \,\, , \,\, \lambda_{m,n}=\pi^2 (m^2+n^2) \,\, , \,\, m,n \in \mathbb{N} \, .
\end{equation*}
Prescribing $f(x_1,x_2)=(2\pi^2)^s \sin(\pi x_1)\sin(\pi x_2)$, from \eqref{eq:spectral-fractional-laplacian} we have that
\begin{equation*}
 u(x_1,x_2) = \sin(\pi x_1) \sin(\pi x_2) \, .
\end{equation*}
A uniform mesh of $\mathbb{Q}_1$ elements is used on $\Omega_\square$ consisting of 66,049 degrees of freedom.
We set $\mathcal{Y} = 10$ and $s=0.25$. For both experiments and each case of number of ranks,
the problem was executed 5 times and the resulting runtimes averaged in order to minimize the potential effects of runtime variability.

The first test measures the strong scaling performance. We fix $\mathcal{K} = 32,000$. The number of MPI processes used to solve the problem is then increased from 1 up to 32. \Cref{fig:a} illustrates
the speedup of the algorithm compared with the ideal linear speedup curve. At 32 MPI processes, the algorithm maintains a speedup of nearly 26.

The weak scaling properties of our method are explored next. Here we scale the amount of
computational work linearly with the number of MPI processes and examine the efficiency of the algorithm, i.e., does the total wall time remain
constant for each parallel run?
We set $\mathcal K = 32,000N$, where $N$ is the number of MPI processes launched. The parallel efficiency can then be calculated as a ratio of the runtime for $N$ MPI processes to
the runtime for a single MPI process. \Cref{fig:b} shows the parallel efficiency curve.
An efficiency in excess of 100\% is maintained with $N=32$.

\begin{figure}[tbhp]
  \centering
  \subfloat[Strong scaling.]{\label{fig:a}
    \begin{tikzpicture}[scale=0.6, transform shape]
      \begin{axis}[
        xlabel={Number of MPI Processes},
        ylabel={Speedup},
        xmin=0, xmax=36,
        ymin=0, ymax=36,
        xtick={0,4,8,12,16,20,24,28,32},
        ytick={0,4,8,12,16,20,24,28,32},
        legend pos=north west,
        ymajorgrids=true,
        grid style=dashed,
      ]
      \addplot[
      color=black,
      mark=none,
      ]
      coordinates {
       (1,1)(2,2)(4,4)(8,8)(12,12)(16,16)(20,20)(24,24)(28,28)(32,32)
      };
      \legend{Ideal Speedup}

     \addplot[
       color=blue,
       mark=square,
     ]
     coordinates {
     (1,1)(2,2.22378484)(4,4.289640737)(8,8.137290679)(12,11.78984085)(16,15.38840064)(20,18.82800932)(24,21.86292239)(28,24.18209027)(32,25.85442726)
     };
     \addlegendentry{Real Speedup}
      \end{axis}
    \end{tikzpicture}
  }
\subfloat[Weak scaling.]{\label{fig:b}
\begin{tikzpicture}[scale=0.6, transform shape]
    \begin{axis}[
      xlabel={Number of MPI Processes},
      ylabel={Efficiency},
      xmin=0, xmax=36,
      ymin=0.92, ymax=1.175,
      xtick={0,4,8,12,16,20,24,28,32},
      ytick={0.925,0.95,0.975,1.00,1.025,1.05,1.075,1.10,1.125,1.15},
      legend pos=north west,
      ymajorgrids=true,
      grid style=dashed,
    ]

    \addplot[
      color=blue,
       mark=square,
     ]
     coordinates {
     (1,1)(2,1.140954176)(4,1.109088825)(8,1.099942595)(12,1.098741258)(16,1.094661557)(20,1.09290542)(24,1.081872548)(28,1.059873249)(32,1.022540024)
     };
     \legend{Parallel Efficiency}
   \end{axis}
  \end{tikzpicture}
}
\caption{Parallel performance for a simple test case.}
\label{fig:par-perf}
\end{figure}
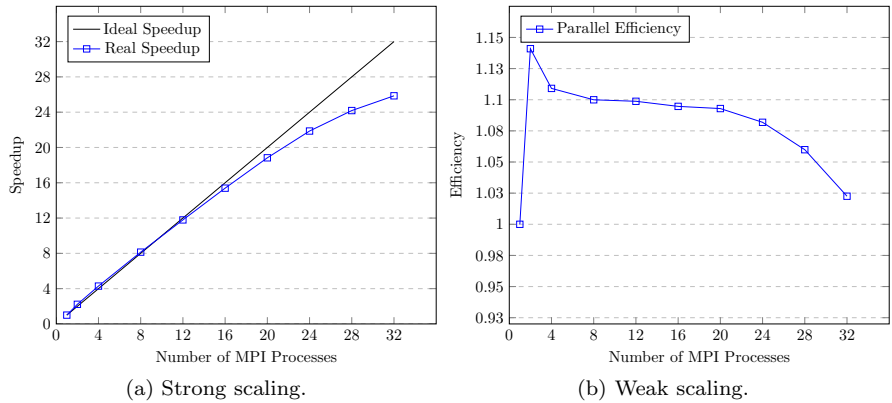

\section{Conclusion}
\label{sec:conclusion}

We introduced and analyzed an exact diagonalization technique for the solution
of the spectral fractional Laplacian using a FEM based on the Caffarelli-Silvestre extension 
\cite{caffarelli2007extension}. This technique builds on \cite{banjai2019tensor} which used \textit{hp}-FEM
to approximate the eigenpairs of the eigenvalue problem in the extended dimension. Error estimates measured in the
$L^2(\Omega)$ norm were derived. Numerical experiments validated the theoretical convergence rate. The parallel performance of the new technique was measured by two
scaling benchmarks, which demonstrate that the algorithm is able to take advantage of the
parallelizability inherent in the numerical scheme.

\appendix
\section{Computation Details and Reproducibility}
\label{sec:software-version}

All numerical experiments were conducted using a single node with dual sockets using
Intel Xeon Gold 6246R processors running at 3.40 GHz with a total of 628 GB of RAM.
The node was running Fedora release 36 configured with Linux kernel 6.2.13. All libraries and source
code were compiled using GCC version 12.2.1. The source code for \texttt{deal.ii} and its dependencies
were built using the \texttt{candi} system, available at
\texttt{\url{https://github.com/dealii/candi}}.
Versions for all libraries used are given in Table \ref{tab:libs}.

\begin{table}[tbhp]
\footnotesize
\caption{Library versions.}\label{tab:libs}
\begin{center}
\begin{tabular}{cc|cc}
OpenMPI   & 4.1.5    & NetCDF           & 4.7.4 \\
METIS     & 5.1.0    & Numdiff          & 5.9.0 \\
deal.ii   & 9.4.2-r3 & Oce-OCE          & 0.18.3 \\
adloc     & 2.7.3    & OpenBLAS         & 0.3.17 \\ \hline
ARPACK-NG & 3.8.0    & P4est            & 2.3.2 \\
assimp    & 4.1.0    & ParMETIS         & 4.0.3 \\
astyle	   & 2.04     & PETSc            & 3.16.4 \\
boost     & 1.63.0   & Scalapack        & 2.1.0 \\ \hline
bzip2     & 1.0.6    & SLEPc            & 3.16.2 \\
Cmake     & 3.20.5   & Sundials         & 5.7.0 \\
Ginkgo	   & 1.4.0    & SuperLU Dist     & 5.1.2 \\
Gmsh      & 4.8.4    & Symengine        & 0.8.1 \\ \hline
GSL       & 2.6      & Trilinos Release & 13.2.0 \\
HDF5      & 1.10.8   & Zlib             & 1.2.8 \\
Mumps     & 5.4.0.5  &                  & \\
\end{tabular}
\end{center}
\end{table}

The source code for the numerical algorithm is available at\\
\texttt{\url{https://github.com/shedsaw/Exact-Diagonalization-Tensor-FEM}}.

\section*{Acknowledgments}

Thanks to  Andrea Bonito (TAMU) for his help on the interpretation of the numerical scheme as
a quadrature formula and to Matthias Maier (TAMU) for lending us his expert knowledge of the \texttt{deal.ii} library.

This work is supported by NSF Grant No. 2111228.


\begin{thebibliography}{10}

\bibitem{abramowitz1988handbook}
{\sc M.~Abramowitz, I.~A. Stegun, and R.~H. Romer}, {\em Handbook of
  mathematical functions with formulas, graphs, and mathematical tables}, 1988.

\bibitem{aceto2017rational}
{\sc L.~Aceto and P.~Novati}, {\em Rational approximation to the fractional
  laplacian operator in reaction-diffusion problems}, SIAM Journal on
  Scientific Computing, 39 (2017), pp.~A214--A228.

\bibitem{adams2003sobolev}
{\sc R.~A. Adams and J.~J. Fournier}, {\em Sobolev spaces}, Elsevier, 2003.

\bibitem{ainsworth2018frac}
{\sc M.~Ainsworth and C.~Glusa}, {\em Hybrid finite element--spectral method
  for the fractional {L}aplacian: Approximation theory and efficient solver},
  SIAM Journal on Scientific Computing, 40 (2018), pp.~A2383--A2405.

\bibitem{dealII95}
{\sc D.~Arndt, W.~Bangerth, M.~Bergbauer, and et~al.}, {\em The deal. {II}
  library, version 9.5}, J. Numer. Math., 31 (2023), pp.~231--246.

\bibitem{dealii2019design}
{\sc D.~Arndt, W.~Bangerth, D.~Davydov, T.~Heister, L.~Heltai, M.~Kronbichler,
  M.~Maier, J.-P. Pelteret, B.~Turcksin, and D.~Wells}, {\em The {deal.II}
  finite element library: Design, features, and insights}, Computers \&
  Mathematics with Applications, 81 (2021), pp.~407--422.

\bibitem{MR0115096}
{\sc A.~V. Balakrishnan}, {\em Fractional powers of closed operators and the
  semigroups generated by them}, Pacific J. Math., 10 (1960), pp.~419--437.

\bibitem{banjai2019tensor}
{\sc L.~Banjai, J.~M. Melenk, R.~H. Nochetto, E.~Ot{\'a}rola, A.~J. Salgado,
  and C.~Schwab}, {\em Tensor {FEM} for spectral fractional diffusion},
  Foundations of Computational Mathematics, 19 (2019), pp.~901--962.

\bibitem{MR4530200}
{\sc L.~Banjai, J.~M. Melenk, and C.~Schwab}, {\em Exponential convergence of
  {$hp$} {FEM} for spectral fractional diffusion in polygons}, Numer. Math.,
  153 (2023), pp.~1--47.

\bibitem{bateman1953higher}
{\sc H.~Bateman}, {\em Higher transcendental functions [volumes i-iii]},
  vol.~3, McGRAW-HILL book company, 1953.

\bibitem{MR1192782}
{\sc M.~S. Birman and M.~Z. Solomjak}, {\em Spectral theory of selfadjoint
  operators in {H}ilbert space}, Mathematics and its Applications (Soviet
  Series), D. Reidel Publishing Co., Dordrecht, 1987.
\newblock Translated from the 1980 Russian original by S. Khrushch\"{e}v and V.
  Peller.

\bibitem{bonito2018numerical}
{\sc A.~Bonito, J.~P. Borthagaray, R.~H. Nochetto, E.~Ot{\'a}rola, and A.~J.
  Salgado}, {\em Numerical methods for fractional diffusion}, Computing and
  Visualization in Science, 19 (2018), pp.~19--46.

\bibitem{bonito2021numerical}
{\sc A.~Bonito and M.~Nazarov}, {\em Numerical simulations of surface
  quasi-geostrophic flows on periodic domains}, SIAM Journal on Scientific
  Computing, 43 (2021), pp.~B405--B430.

\bibitem{bonito2015numerical}
{\sc A.~Bonito and J.~Pasciak}, {\em Numerical approximation of fractional
  powers of elliptic operators}, Mathematics of Computation, 84 (2015),
  pp.~2083--2110.

\bibitem{bonito2017numerical}
{\sc A.~Bonito and J.~E. Pasciak}, {\em Numerical approximation of fractional
  powers of regularly accretive operators}, IMA Journal of Numerical Analysis,
  37 (2017), pp.~1245--1273.

\bibitem{bonito2020electroconvection}
{\sc A.~Bonito and P.~Wei}, {\em Electroconvection of thin liquid crystals:
  Model reduction and numerical simulations}, Journal of Computational Physics,
  405 (2020), p.~109140.

\bibitem{MR4132117}
{\sc J.~P. Borthagaray, W.~Li, and R.~H. Nochetto}, {\em Linear and nonlinear
  fractional elliptic problems}, in 75 years of mathematics of computation,
  vol.~754 of Contemp. Math., Amer. Math. Soc., 2020, pp.~69--92.

\bibitem{cabre2010positive}
{\sc X.~Cabr{\'e} and J.~Tan}, {\em Positive solutions of nonlinear problems
  involving the square root of the {L}aplacian}, Advances in Mathematics, 224
  (2010), pp.~2052--2093.

\bibitem{caffarelli2007extension}
{\sc L.~Caffarelli and L.~Silvestre}, {\em An extension problem related to the
  fractional {L}aplacian}, Communications in partial differential equations, 32
  (2007), pp.~1245--1260.

\bibitem{chen2020efficient}
{\sc S.~Chen and J.~Shen}, {\em An efficient and accurate numerical method for
  the spectral fractional laplacian equation}, Journal of Scientific Computing,
  82 (2020), p.~17.

\bibitem{cusimano2018space}
{\sc N.~Cusimano and L.~Gerardo-Giorda}, {\em A space-fractional monodomain
  model for cardiac electrophysiology combining anisotropy and heterogeneity on
  realistic geometries}, Journal of Computational Physics, 362 (2018),
  pp.~409--424.

\bibitem{cusimano2021space}
{\sc N.~Cusimano, L.~Gerardo-Giorda, and A.~Gizzi}, {\em A space-fractional
  bidomain framework for cardiac electrophysiology: 1d alternans dynamics},
  Chaos: An Interdisciplinary Journal of Nonlinear Science, 31 (2021).

\bibitem{MR4189291}
{\sc M.~D'Elia, Q.~Du, C.~Glusa, M.~Gunzburger, X.~Tian, and Z.~Zhou}, {\em
  Numerical methods for nonlocal and fractional models}, Acta Numer., 29
  (2020), pp.~1--124.

\bibitem{di2012hitchhikers}
{\sc E.~Di~Nezza, G.~Palatucci, and E.~Valdinoci}, {\em Hitchhiker's guide to
  the fractional sobolev spaces}, Bulletin des sciences math{\'e}matiques, 136
  (2012), pp.~521--573.

\bibitem{NIST:DLMF}
{\em {\it NIST Digital Library of Mathematical Functions}}.
\newblock \url{https://dlmf.nist.gov/}, Release 1.1.11 of 2023-09-15,
  \url{https://dlmf.nist.gov/}.
\newblock F.~W.~J. Olver, A.~B. {Olde Daalhuis}, D.~W. Lozier, B.~I. Schneider,
  R.~F. Boisvert, C.~W. Clark, B.~R. Miller, B.~V. Saunders, H.~S. Cohl, and
  M.~A. McClain, eds.

\bibitem{MR4122489}
{\sc B.~Duan, R.~D. Lazarov, and J.~E. Pasciak}, {\em Numerical approximation
  of fractional powers of elliptic operators}, IMA J. Numer. Anal., 40 (2020),
  pp.~1746--1771.

\bibitem{ern2004theory}
{\sc A.~Ern and J.-L. Guermond}, {\em Theory and practice of finite elements},
  vol.~159, Springer, 2004.

\bibitem{evans2022partial}
{\sc L.~C. Evans}, {\em Partial differential equations}, vol.~19, American
  Mathematical Society, 2022.

\bibitem{gatto2015numerical}
{\sc P.~Gatto and J.~S. Hesthaven}, {\em Numerical approximation of the
  fractional {L}aplacian via $hp$ hp-finite elements, with an application to
  image denoising}, Journal of Scientific Computing, 65 (2015), pp.~249--270.

\bibitem{harizanov2020analysis}
{\sc S.~Harizanov, R.~Lazarov, S.~Margenov, P.~Marinov, and J.~Pasciak}, {\em
  Analysis of numerical methods for spectral fractional elliptic equations
  based on the best uniform rational approximation}, Journal of Computational
  Physics, 408 (2020), p.~109285.

\bibitem{harizanov2018optimal}
{\sc S.~Harizanov, R.~Lazarov, S.~Margenov, P.~Marinov, and Y.~Vutov}, {\em
  Optimal solvers for linear systems with fractional powers of sparse spd
  matrices}, Numerical Linear Algebra with Applications, 25 (2018), p.~e2167.

\bibitem{hofreither2020unified}
{\sc C.~Hofreither}, {\em A unified view of some numerical methods for
  fractional diffusion}, Computers \& Mathematics with Applications, 80 (2020),
  pp.~332--350.

\bibitem{ivrii2016100}
{\sc V.~Ivrii}, {\em 100 years of {W}eyl’s law}, Bulletin of Mathematical
  Sciences, 6 (2016), pp.~379--452.

\bibitem{kufner1984define}
{\sc A.~Kufner and B.~Opic}, {\em How to define reasonably weighted {S}obolev
  spaces}, Commentationes Mathematicae Universitatis Carolinae, 25 (1984),
  pp.~537--554.

\bibitem{MR4043885}
{\sc A.~Lischke, G.~Pang, M.~Gulian, and et~al.}, {\em What is the fractional
  {L}aplacian? {A} comparative review with new results}, J. Comput. Phys., 404
  (2020), pp.~109009, 62.

\bibitem{lunardi2009interpolation}
{\sc A.~Lunardi et~al.}, {\em Interpolation theory}, vol.~9, Edizioni della
  normale Pisa, 2009.

\bibitem{matsuki1993note}
{\sc M.~Matsuki and T.~Ushijima}, {\em A note on the fractional powers of
  operators approximating a positive definite selfadjoint operator}, Journal of
  the Faculty of Science, the University of Tokyo. Sect. 1 A, Mathematics, 40
  (1993), pp.~517--528.

\bibitem{metzler2000}
{\sc R.~Metzler and J.~Klafter}, {\em The random walk's guide to anomalous
  diffusion: a fractional dynamics approach}, Physics Reports, 339 (2000),
  pp.~1--77.

\bibitem{nochetto2015pde}
{\sc R.~H. Nochetto, E.~Ot{\'a}rola, and A.~J. Salgado}, {\em A {PDE} approach
  to fractional diffusion in general domains: a priori error analysis},
  Foundations of Computational Mathematics, 15 (2015), pp.~733--791.

\bibitem{ogata2005numerical}
{\sc H.~Ogata}, {\em A numerical integration formula based on the bessel
  functions}, Publications of the Research Institute for Mathematical Sciences,
  41 (2005), pp.~949--970.

\bibitem{oldham2009atlas}
{\sc K.~B. Oldham, J.~Myland, and J.~Spanier}, {\em An atlas of functions: with
  equator, the atlas function calculator}, Springer, 2009.

\bibitem{reed1981functional}
{\sc M.~Reed and B.~Simon}, {\em I: Functional analysis}, vol.~1, Academic
  press, 1981.

\bibitem{stinga2010extension}
{\sc P.~R. Stinga and J.~L. Torrea}, {\em Extension problem and {H}arnack's
  inequality for some fractional operators}, Communications in Partial
  Differential Equations, 35 (2010), pp.~2092--2122.

\bibitem{strauss2007partial}
{\sc W.~A. Strauss}, {\em Partial differential equations: An introduction},
  John Wiley \& Sons, 2007.

\bibitem{thomee2007galerkin}
{\sc V.~Thom{\'e}e}, {\em Galerkin finite element methods for parabolic
  problems}, vol.~25, Springer Science \& Business Media, 2007.

\bibitem{vabishchevich2015numerically}
{\sc P.~N. Vabishchevich}, {\em Numerically solving an equation for fractional
  powers of elliptic operators}, Journal of Computational Physics, 282 (2015),
  pp.~289--302.

\bibitem{zhang2015many}
{\sc Z.~Zhang}, {\em How many numerical eigenvalues can we trust?}, Journal of
  Scientific Computing, 65 (2015), pp.~455--466.

\end{thebibliography}

\end{document}